\documentclass[12pt, letterpaper]{article}
\usepackage{amsmath}
\usepackage{amsthm}
\usepackage{amsfonts}
\usepackage{amssymb}
\usepackage{color}
\usepackage{enumitem}
\newtheorem{theorem}{Theorem}[section]

\newtheorem{lemma}{Lemma}[section]
\newtheorem{definition}{Definition}[section]
\newtheorem{proposition}{Proposition}[section]
\newtheorem{example}{Example}
\newtheorem{claim}{Claim}
\newtheorem{remark}{Remark}[section]

\allowdisplaybreaks
\newcommand\blfootnote[1]{%
  \begingroup
  \renewcommand\thefootnote{}\footnote{#1}%
  \addtocounter{footnote}{-1}%
  \endgroup
}		

\begin{document}
\title{Monge solutions and uniqueness in multi-marginal optimal transport: weaker conditions on the cost, stronger conditions on the marginals}
\author{BRENDAN PASS \\	 
	\and 
	ADOLFO VARGAS-JIM\'{E}NEZ}
	\date{}
\maketitle
\begin{abstract}
\blfootnote{
		\textit{Date:} 14 February, 2022.
\vspace{0.1cm}

BP is pleased to acknowledge support from Natural Sciences and Engineering Research Council of Canada Grant 04658-2018. The work of AVJ was completed in partial fulfillment of the requirements for a doctoral degree in mathematics at the University of Alberta. }
 We establish a general condition on the cost function to obtain uniqueness and Monge solutions in the multi-marginal optimal transport problem, under the assumption that a given collection of the marginals are absolutely continuous with respect to local coordinates. When only the first marginal is assumed to be absolutely continuous, our condition is equivalent to the twist on splitting sets condition found in \cite{Kim}.  In addition, it is satisfied by the special cost functions in our earlier work \cite{PV2020, PV2021}, when absolute continuity is imposed on certain other collections of marginals.  We also present several  new examples of cost functions which violate the twist on splitting sets condition but satisfy the new condition introduced here; we therefore obtain Monge solution and uniqueness results for these cost functions, under  regularity conditions on an appropriate subset of the marginals.
	
\end{abstract}

\section{INTRODUCTION}
Let  $X_{i}$ be smooth manifolds and  $P(X_{i})$ the set of all Borel probability measures on $X_i$, with $i=1, \ldots, m$. Given a continuous real-valued  cost function $c$ on the product space $\prod_{i=1}^{m}X_i$, and compactly supported measures $\mu_i\in P(X_{i})$, the Kantorovich formulation of the multi-marginal  optimal transportation problem asks to minimize
\begin{equation}\label{KP}
\displaystyle \int_{\prod_{i=1}^{m}X_i} c(x_{1},\ldots, x_{m})d\mu,\tag{KP}
 \end{equation}
 \text{over the admissible class}
 \begin{flalign*}
\Pi(\mu_{1},\ldots , \mu_{m})&:= \Big\{  \mu \in P(\prod_{i=1}^{m}X_i): \mu(X_{1}\times \ldots \times X_{i-1}\times A_{i} \times X_{i+1}\times \ldots \times X_{m})=\mu_{i} (A_{i}),\\
 &\qquad\quad\qquad\qquad\text{ for every measurable set}\;\; A_{i}\subseteq X_{i},\quad 1\leq i\leq m \Big\}.
 \end{flalign*} 
On the other hand, in Monge's formulation of the multi-marginal  optimal transportation problem one seeks to minimize
\begin{equation}\label{MP}
 \displaystyle \int_{X_{1}} c(x_{1},T_{2}x_{1},\ldots, T_{m}x_{1})d\mu_{1},\tag{MP}
 \end{equation}
 over its admissible class: the set of all $(m-1)$-tuples of maps $(T_{2}, \ldots, T_{m})$ satisfying the constraint  $(T_{i})_{\sharp} \mu_{1}=\mu_{i}$ for every $i=2, \ldots, m$. Here, $(T_{i})_{\sharp} \mu_{1}$ denotes the \textit{image measure} of $\mu_{1}$ through $T_{i}$, which is defined as $(T_{i})_{\sharp} \mu_{1}(A) = \mu_{1}(T_{i}^{-1}(A))$, for any Borel set $A\subseteq X_{i}$. The admissible class in (\ref{MP}) can be seen as a subclass of $\Pi(\mu_{1},\ldots , \mu_{m})$, as for each  $(m-1)$-tuple $(T_{2}, \ldots, T_{m})$ satisfying the image measure constraint in (\ref{MP}), we get $\mu=(Id,T_{2}, \ldots, T_{m})_{\sharp} \mu_{1}\in \Pi(\mu_{1},\ldots , \mu_{m})$ and 
 $$\displaystyle \int_{\prod_{i=1}^{m}X_i} c(x_{1},\ldots, x_{m})d\mu=\displaystyle \int_{X_{1}} c(x_{1},T_{2}x_{1},\ldots, T_{m}x_{1})d\mu_{1}.$$
This fact lets us interpret (\ref{KP}) as a relaxation of (\ref{MP}).

The classical optimal transport (the case $m=2$) has an abundant literature, reflecting its natural connections with different areas of mathematics and wide variety of applications. See, for instance, \cite{Filippo}\cite{Villani}\cite{Villani2} or \cite{AG2013} for an overview. In the multi-marginal optimal transport (the case $m\geq 3$), diverse applications have also been emerging, among them, matching in economics \cite{CE2010}\cite{CMN2010}\cite{Pass6}, density functional theory in computation \cite{BPG2012} \cite{CFC2013}, and more recently, interpolating among distributions in machine learning and statistics \cite{BM2021}\cite{ZP2019}. The reader is also referred to \cite{Pass3} for an overview. 

One of the foundational results in the classical optimal transport is that, under a twist condition on $c$ (injectivity of the map $x_{2}\mapsto D_{x_{1}}c(x_{1},x_{2})$,  for each fixed $x_{1}\in X_1$) and assuming $\mu_{1}$ absolutely continuous with respect to local coordinates, there exists a unique solution to \eqref{KP} and it is induced by a map \cite{Brenier}\cite{McCann}\cite{Gangbo1}\cite{Gangbo2}. In the multi-marginal setting, a condition playing an analogous role was discovered in \cite{Kim}; this condition was called
	\textit{twist on $c$-splitting sets} and  states that for every $x_1\in X_1$ fixed,  the map $(x_2, \ldots, x_m)\mapsto D_{x_1}c(x_1, x_2, \ldots, x_m)$ is injective on $c$-splitting sets (see definition \ref{Art3:44}). The main result in \cite{Kim} is then that whenever $\mu_1$ is absolutely continuous with respect to local coordinates, and $c$ twisted on $c$-splitting sets, the solution $\gamma$ to \eqref{KP} is unique and induced by a graph.  This encapsulates the results for specific costs, or costs satisfying certain conditions, found in \cite{C2010,  GS1998, H2002, KP2015, Pass4, Pass6, PPV}.  Unlike its two marginal analogue (the classical twist condition), the twist on $c$-splitting sets is very strong; there are many examples of cost functions for which it fails, and for which non-unique, non-Monge type solutions exist \cite{Pass1}\cite{Carlier2}\cite{Pass13}\cite{F2019}\cite{GKR2019}\cite{ColomboDePascaleDiMarino15}.  It is, however, the most general known condition guaranteeing the unique Monge structure of solutions, and it seems unlikely that there is a significantly weaker condition on $c$ under which these hold for \emph{all} choices of marginals $\mu_1,...,\mu_m$ with $\mu_1$ absolutely continuous.  
	
	However, our recent work \cite{PV2020,PV2021} uncovered examples of cost functions which violate the twist on splitting sets condition, but for which we were able to establish Monge solution and uniqueness results; the trade-off is that we had to assume regularity of certain subsets of the marginals, rather than only $\mu_1$.  This naturally motivates the pursuit of a general condition on $c$, under which solutions to \eqref{KP} are of Monge type and unique, for any collection of marginals $\mu_1,...,\mu_m$ with $\mu_i$ absolutely continuous for all $i$ in a given subset of $\{1,2,...,m\}$. The purpose of this paper is to develop such a condition.

\par
Our condition is formulated in terms of $c$-splitting functions (see Definition \ref{Art3:21}) and the points where some of them are differentiable (the ones corresponding to the marginals different than $\mu_1$ where regularity is needed).    More specifically, we require  the mapping $(x_2, \ldots, x_m)\mapsto D_{x_1}c(x_1, x_2, \ldots, x_m)$  to be injective on special subsets generated by $c$-splitting sets and their associated Borel functions (see Definition \ref{Art3:20}). This condition ensures Monge structure and uniqueness of the optimal elements in $\Pi(\mu_{1},\ldots , \mu_{m})$, as we shall see in our main result (Theorem 3.1). This condition reduces to the twist on splitting sets condition in the special case when only regularity of $\mu_1$ is assumed, but reaches substantially beyond it in general.  Aside from including the cost functions in \cite{PV2020} and \cite{PV2021}, our condition applies to a wide variety of new costs, as we illustrate with several examples.


One essential aspect  of the version of the twist condition presented on this work is the dependence on  $c$-splitting functions of the sets where the map $(x_2, \ldots, x_m)\mapsto D_{x_1}c(x_1, x_2, \ldots, x_m)$ is injective (unlike the twist on $c$-splitting sets condition where such map is injective on splitting sets with no dependency on  $c$-splitting functions). The involvement of $c$-splitting functions allow us to naturally generate several differential equations as the presented in Lemma \ref{Art3:8}, which are key to naturally exploit the structure of a variety of cost functions. This type of approach is possible, in particular, by the incorporation of additional regularity conditions on the marginals. We also establish an equivalent condition to the twist on $c$-splitting sets condition that facilitates the proof of some of the results; this condition focuses on every $m$-tuple of $c$-splitting functions and an associated largest $c$-splitting set, instead of every $c$-splitting set and its associated $c$-splitting functions (see Lemma \ref{Art3:16}). 
 \par
In the next section, we recall and introduce the essential definitions used on this work, as well as some key lemmas. In section 3, we establish and prove our main result. In Section 4, we provide several examples of cost functions satisfying our condition.

 \section{Preliminaries}
 Let us recall some main concepts from \cite{Kim}.
\begin{definition}\label{Art3:21}
A set $S\subseteq \prod_{i=1}^{m}X_i$ is called a $c$-splitting set if there are Borel functions $u_i:X_i \mapsto \mathbb{R}$ such that 
\begin{equation}\label{Art3:6}
\sum_{i=1}^{m}u_i(x_{i})\leq c(x_1, \ldots,x_m)
\end{equation}
for every $(x_1, \ldots, x_m)\in \prod_{i=1}^{m}X_i$, and whenever $(x_1, \ldots, x_m)\in S$ equality holds. The functions $u_1(x_1), \ldots, u_m(x_m)$ are called $c$-splitting functions for $S$.
\end{definition}
\begin{definition}\label{Art3:43}
A set $S\subseteq \prod_{i=1}^{m}X_i$ is called $c$-cyclically monotone if for any finite collection $\left\lbrace (x_{1}^{k}, \ldots, x_{m}^{k}) \right\rbrace_{k=1}^{p}\subseteq S$ we get 
\begin{equation*}
\sum_{k=1}^{p}c(x_{1}^{k}, \ldots, x_{m}^{k})\leq \sum_{k=1}^{p}c(x_{1}^{\sigma_1(k)}, \ldots, x_{m}^{\sigma_m(k)}),
\end{equation*}
for every $\sigma_1, \ldots, \sigma_m\in S_P$, where $S_P$ denotes the set of permutations of $P:=\{ 1, \ldots, p\}$.
\end{definition}
It is straightforward to prove that  any $c$-splitting set is $c$-cyclically monotone. When $m=2$, the converse is true by R{\"u}schendorf theorem \cite{R1996}. The converse for $m\geq 3$, remained an open question until Griessler proved that in fact, every $c$-cyclically monotone set is $c$-splitting \cite{G2018}. In this work, we shall find it convenient to use both definitions interchangeably.
\begin{definition}\label{Art3:44}
Let $c$ be a continuous semi-concave cost function. It is called twisted on $c$-splitting sets, whenever for each fixed $x_1^{0}\in X_1$ and $c$-splitting set $S\subseteq \{x_1^{0}\}\times X_2\times \ldots X_m$, the map 
$$(x_{2}, \ldots, x_{m})\mapsto D_{x_{1}}c(x_{1}^{0},x_2,\ldots, x_{m}) $$ is injective on
the subset of  $S$ where $D_{x_1}c(x_1^{0}, x_2, \ldots, x_m)$ exists. 
\end{definition}
\begin{remark}\label{Art3:45}
The main result in \cite{Kim} establish that if $c$ is twisted on $c$-splitting sets, then every solution to $(\ref{KP})$ is induced by a map, whenever $\mu_1$ is absolutely continuous with respect to local coordinates. 
\end{remark}
 \par
A classical duality theorem of Kellerer \cite{Kellerer}, automatically connects Definitions \ref{Art3:21} and \ref{Art3:43} with the optimal measures $\gamma$ in \eqref{KP}. From now on, $spt(\gamma)$ denotes the support of $\gamma$.
\begin{lemma}\label{Art3:2}
A measure $\gamma \in \Pi(\mu_{1},\ldots , \mu_{m})$ is optimal in \eqref{KP} if and only if $spt(\gamma)$  is a $c$-splitting set. 
\end{lemma}
\begin{remark}\label{Art3:3}
The above lemma guarantees the existence of an $m$-tuple $(u_1, \ldots, u_m)$ of $c$-splitting functions to $spt(\gamma)$, for every optimal measure  $\gamma \in \Pi(\mu_{1},\ldots , \mu_{m})$ in \eqref{KP}. A key fact for this work is that these $c$-splitting functions can be taken to be $c$-conjugate \cite{GS1998}\cite{Pass4}. More specifically,  for each $i$,
\begin{equation}\label{Art3:1}
u_{i}(x_{i})= inf_{x_j\in X_{j}, j\neq i}\Big ( c(x_{1},\ldots,x_{m}) -\sum_{j\neq i}u_{j}(x_{j})\Big ).
\end{equation} 
\end{remark}
Let us finish this section with a convenient lemma, which will reduce some of the technical details of the results in this work. For this, recall that given an open set $D$ and a semi-concave function $f:D\subseteq\mathbb{R}^{n}\mapsto \mathbb{R}$, with semiconcavity constant $\lambda$, the superdifferential of $f$ with respect to a given $x\in A$ fixed is defined as the set 
$$\partial f(x)= \left\lbrace z\in \mathbb{R}^{n}: f(y)-f(x)\leq z\cdot (y-x) + \lambda\mid y-x\mid^{2} \;\; \forall y\in D\right\rbrace.$$
 If $f$ is defined in a smooth manifold we keep the same definition by using local coordinates. 
It can be proved that $\partial f(x)$ is nonempty for every $x\in D$ and $Df(x)$ exists if and only if $\partial f(x)$ is a singleton.
\begin{lemma}\label{Art3:8}
  Let $c$ be a continuous semi-concave cost function, and $u_i:X_i \mapsto \mathbb{R}$  Borel functions, $i\in \{1, \ldots, m\}$, satisfying the inequality condition in  \eqref{Art3:6}. Let $(x_1^{0}, \ldots, x_m^{0})\in \prod_{i=1}^{m}X_i$ such that
  \begin{equation}\label{Art3:7}
\sum_{i=1}^{m}u_i(x_{i}^{0})=c(x_1^{0}, \ldots,x_m^{0}).
  \end{equation}
  If there exists $k\in \{1, \ldots, m\}$ such that $Du_{k}(x_k^{0})$ exists, then $D_{x_k}c(x_{1}^{0}, \ldots, x_{m}^{0})$ exists and 
  \begin{equation*}
  Du_{k}(x_k^{0})=D_{x_k}c(x_{1}^{0}, \ldots, x_{m}^{0}).
  \end{equation*}
\end{lemma}
\begin{proof}
Since $c$ is semi-concave, the map $x_{k}\mapsto c(x_{1}^{0}, \ldots,x_{k-1}^{0},x_{k}, x_{k+1}^{0}, \ldots, x_{m}^{0})$ is semi-concave. Then $\partial_{x_k} c(x_{1}^{0}, \ldots,x_{k-1}^{0},x_{k}, x_{k+1}^{0}, \ldots, x_{m}^{0})$ is nonempty for every $x_{k}\in X_{k}$ fixed, where $\partial_{x_k} c(x_{1}^{0}, \ldots,x_{k-1}^{0},x_{k}, x_{k+1}^{0}, \ldots, x_{m}^{0})$ denotes the superdifferential of $c$ with respect to $x_{k}$. Using \eqref{Art3:7}, it follows that 
$$\partial_{x_k} c(x_{1}^{0}, \ldots, x_{m}^{0})\subseteq \partial u_{k}(x_k^{0})=\{Du_{k}(x_k^{0})\}.$$
Thus, $\partial_{x_k} c(x_{1}^{0}, \ldots, x_{m}^{0})$ is a singleton, which implies that $D_{x_k} c(x_{1}^{0}, \ldots, x_{m}^{0})$ exists and $Du_{k}(x_k^{0})=D_{x_k}c(x_{1}^{0}, \ldots, x_{m}^{0})$, completing the proof.
\end{proof}
\subsection{Essential definitions and preliminary results}
Here we establish the main concepts used on this work. For this, we  first introduce some convenient notations. Assume $\{k_{i}\}_{i=1}^{r}\subseteq \{2, \ldots, m\}$, with $k_{1}< k_{2}<\ldots < k_{r}$. 
\begin{itemize}
\item Let $S\subseteq \prod_{i=1}^{m}X_i$ be a  $c$-splitting set and $(u_{1},\ldots, u_{m})$ an $m-$tuple  of c-splitting functions for $S$. Given $x_1^{0} \in \pi_{1}(S)$, where $\pi_{1}$ is the canonical projection from $\prod_{i=1}^{m}X_{i}$ to $X_1$, we define 
\begin{align*}
W_{x_1^{0}k_1\ldots k_r}(u_{1}, \ldots, u_{m},S)&:=\Big\{ (x_{2}, \ldots, x_{m})\in \prod_{i=2}^{m}X_i:  (x_{1}^{0},x_2, \ldots, x_{m})\in S \;\text{and}\\
&\qquad \qquad \qquad\qquad \qquad \text{$Du_{k_{i}}(x_{k_{i}})$}\;\;\text{exists for each $i=1, \ldots r$}\Big\}.
\end{align*}
\item For a given $m$-tuple of Borel functions $(u_{1}^{\prime},\ldots, u_{m}^{\prime})$ satisfying inequality \eqref{Art3:6}, and $x_1^{0}\in X_1$, we define
\begin{align*}
M_{x_1^{0}k_1\ldots k_r}(u_1^{\prime}, \ldots, u_m^{\prime})&:=\Big\{ (x_{2}, \ldots, x_{m})\in \prod_{i=2}^{m}X_i: \text{$Du_{k_{i}}^{\prime}(x_{k_{i}})$ exists for each }\\
&\qquad \qquad\quad i=1,\ldots,r\;\;\text{and}\;u_1^{\prime}(x_1^{0})+\sum_{i=2}^{m}u_i^{\prime}(x_i)=c(x_1^{0}, x_2, \ldots,x_m)\Big\}.
\end{align*}
\end{itemize}
From now on, if there is not danger of confusion, we will write $W_{x_1^{0}k_1\ldots k_r}$ and  $M_{x_1^{0}k_1\ldots k_r}$ for $W_{x_1^{0}k_1\ldots k_r}(u_{1}, \ldots, u_{m},S)$ and $M_{x_1^{0}k_1\ldots k_r}(u_1^{\prime}, \ldots, u_m^{\prime})$ respectively.
\begin{remark}\label{Art3:16}
Note that $W_{x_1^{0}k_1\ldots k_r}(u_{1}, \ldots, u_{m},S)\subseteq M_{x_1^{0}k_1\ldots k_r}(u_1, \ldots, u_m)$, for any $c$-splitting set $S$, $m$-tuple  of c-splitting functions $(u_{1},\ldots, u_{m})$ for $S$ and $x_1^{0} \in \pi_{1}(S)$. Hence, for any fixed $(u_1, \ldots, u_m)$ satisfying inequality \eqref{Art3:6} and $x_{1}^{0}\in X_1$, we get
$$\bigcup_{S\in \mathcal{F}}W_{x_1^{0}k_1\ldots k_r}(u_{1}, \ldots, u_{m},S)\subseteq M_{x_1^{0}k_1\ldots k_r}(u_1, \ldots, u_m),$$
where
\begin{align*}
\mathcal{F}&:=\Big\{  S\subseteq \prod_{i=1}^{m}X_i: x_1^{0}\in \pi_{1}(S)\;\; \text{and $S$ is a splitting set having $(u_1, \ldots,u_m)$}\\
& \quad\quad\text{as $c$-splitting functions}\Big\}.
\end{align*}
On the other hand, for any $(x_2, \ldots, x_m)\in M_{x_1^{0}k_1\ldots k_r}(u_1, \ldots, u_m)$, the singleton $\bar{S}=\{x_1^{0}, x_2, \ldots, x_m\}$ is trivially a $c$-splitting set satisfying $(x_2, \ldots, x_m)\in W_{x_1^{0}k_1\ldots k_r}(u_{1}, \ldots, u_{m},\bar{S})$ with $\bar{S}\in \mathcal{F}$. This immediately implies 
$$\bigcup_{S\in \mathcal{F}}W_{x_1^{0}k_1\ldots k_r}(u_{1}, \ldots, u_{m},S)= M_{x_1^{0}k_1\ldots k_r}(u_1, \ldots, u_m).$$
\end{remark}
\begin{definition}\label{Art3:20}
Let $c$ be a continuous semi-concave cost function, and let $\{k_{i}\}_{i=1}^{r}\subseteq \{2, \ldots, m\}$, with $k_{1}< k_{2}<\ldots < k_{r}$.
We say $c$ is twisted on $c$-splitting sets  with respect to the variables $x_{1}, x_{k_{1}},\ldots, x_{k_{r}}$, if for each $c$-splitting set $S\subseteq \prod_{i=1}^{m}X_i$ and $m$-tuple $(u_{1},\ldots, u_{m})$ of c-splitting functions for $S$,  the map 
$$(x_{2}, \ldots, x_{m})\mapsto D_{x_{1}}c(x_{1}^{0},x_2,\ldots, x_{m}) $$ is injective on
the subset of  $W_{x_1^{0}k_1\ldots k_r}$ where $D_{x_1}c(x_1^{0}, x_2, \ldots, x_m)$ exists, for each fixed $x_1^{0} \in \pi_{1}(S)$ satisfying  $W_{x_1^{0}k_1\ldots k_r}\neq \emptyset$.
\end{definition}
\begin{remark}
Note that Definition \ref{Art3:44} is equivalent to $c$ being twisted on $c$-splitting sets with respect to the variable $x_{1}$. Hence, our main result (Theorem \ref{Main Theorem}), generalizes the main result in \cite{Kim} (see Remark \ref{Art3:45}).
\end{remark}
We now proceed to prove a lemma, which provides an alternative way to check the condition above.
\begin{lemma}\label{Art3:10}
Let $c$ be a continuous, semi-concave cost function. Let $\{k_{i}\}_{i=1}^{r}\subseteq \{2, \ldots, m\}$, with $k_{1}< k_{2}<\ldots < k_{r}$. The cost $c$ is twisted on $c$-splitting sets  with respect to the variables $x_{1}, x_{k_{1}},\ldots, x_{k_{r}}$ if and only if 
for every $m$-tuple of Borel functions $(u_1,\ldots, u_m)$ satisfying inequality \eqref{Art3:6} and for every $x_1^{0}\in X_1$ with $M_{x_1^{0}k_1\ldots k_r}\neq \emptyset$, we get that the map 
$$(x_{2}, \ldots, x_{m})\mapsto D_{x_{1}}c(x_{1}^{0},x_2,\ldots, x_{m}) $$ is injective on
the subset of  $M_{x_1^{0}k_1\ldots k_r}$ where $D_{x_1}c(x_1^{0}, x_2, \ldots, x_m)$ exists.  
\end{lemma}
\begin{proof}
The converse is straightforward, as for every $c$-splitting set $S\subseteq \prod_{i=1}^{m}X_i$ and $m$-tuple $(u_1, \ldots,u_m)$ of $c$-splitting functions for $S$, we have $W_{x_1^{0}k_1\ldots k_r}\subseteq M_{x_1^{0}k_1\ldots k_r}$ for  each fixed $x_1^{0} \in \pi_{1}(S)$. Hence, if $W_{x_1^{0}k_1\ldots k_r}\neq \emptyset$ we get $M_{x_1^{0}k_1\ldots k_r}\neq \emptyset$, which implies that the map $(x_{2}, \ldots, x_{m})\mapsto D_{x_{1}}c(x_{1}^{0},x_2,\ldots, x_{m}) $ is injective on the subset of  $M_{x_1^{0}k_1\ldots k_r}$ where $D_{x_1}c(x_1^{0}, x_2, \ldots, x_m)$ exists, in particular, it is injective on the subset of  $W_{x_1^{0}k_1\ldots k_r}$ where $D_{x_1}c(x_1^{0}, x_2, \ldots, x_m)$ exists; that is,  $c$ is twisted on $c$-splitting sets  with respect to the variables $ x_{1}, x_{k_{1}},\ldots, x_{k_{r}}$. Assume now that the cost $c$ is twisted on $c$-splitting sets  with respect to the variables $x_{1}, x_{k_{1}},\ldots, x_{k_{r}}$. Let $(u_1,\ldots, u_m)$ be an  $m$-tuple of Borel functions satisfying inequality \eqref{Art3:6}, and fix $x_1^{0} \in X_1$. Assume $M_{x_1^{0}k_1\ldots k_r}\neq \emptyset$, and set 
\begin{align*}
S&:=\Big\{ (x_1^{0},x_{2}, \ldots, x_{m})\in \prod_{i=1}^{m}X_i:(x_{2}, \ldots, x_{m}) \in M_{x_1^{0}k_1\ldots k_r}\Big\}\\
&= \Big\{ (x_1^{0},x_{2}, \ldots, x_{m})\in \prod_{i=1}^{m}X_i: u_1(x_1^{0})+\sum_{i=2}^{m}u_i(x_i)=c(x_1^{0}, x_2, \ldots,x_m), \;\text{and}\;\\
&\qquad \qquad \qquad\qquad \qquad\qquad \qquad\quad\;\;\text{$Du_{k_{i}}(x_{k_{i}})$ exists for each $i=1, \ldots r$}\Big\}.
\end{align*}
Clearly, $S$ is a $c$-splitting set,  $ \pi_{1}(S) = \{x_1^{0}\}$ and $W_{x_1^{0}k_1\ldots k_r}=M_{x_1^{0}k_1\ldots k_r}\neq \emptyset$. This immediately implies, by assumption that the map $(x_{2}, \ldots, x_{m})\mapsto D_{x_{1}}c(x_{1}^{0},x_2\ldots, x_{m}) $ is injective on the subset of $M_{x_1^{0}k_1\ldots k_r}$ where $D_{x_1}c(x_1^{0}, x_2, \ldots, x_m)$ exists, completing  the proof of the lemma. \
\end{proof}
\begin{remark}\label{Art3:58}
Note that by Lemma \ref{Art3:8}, if $(u_1,\ldots, u_m)$ is an  $m$-tuple of Borel functions satisfying inequality \eqref{Art3:6} and $Du_1(x_1^{0})$ exists for some
 $x_1^{0}\in X_1$ satisfying $M_{x_1^{0}k_1\ldots k_r}\neq \emptyset$, then the map 
$$(x_{2}, \ldots, x_{m})\mapsto D_{x_{1}}c(x_{1}^{0},x_2,\ldots, x_{m}) $$ is injective on
the subset of  $M_{x_1^{0}k_1\ldots k_r}$   where $D_{x_1}c(x_1^{0}, x_2, \ldots, x_m)$ exists if and only if $M_{x_1^{0}k_1\ldots k_r}$ is a singleton. As we shall see in the next two sections, this fact will be convenient for the proof of our main result (Theorem \ref{Main Theorem}) and  the propositions in Section 4.
\end{remark}
\section{\textbf{Existence and Uniqueness to Monge Problem
}}
We now state and prove our main result.
\begin{theorem}\label{Main Theorem}
Assume the measures $\mu_{1},\mu_{k_1}, \ldots,\mu_{k_r}$ are absolutely continuous with respect to local coordinates, with  $\{k_{i}\}_{i=1}^{r}\subseteq \{2, \ldots, m\}$, $k_{1}< k_{2}<\ldots < k_{r}$. Assume $c$ is twisted on $c$-splitting sets  with respect to the variables $x_{1}, x_{k_{1}},\ldots, x_{k_{r}}$. Then the solution $\gamma$ in (\ref{KP}) is concentrated on a graph of a measurable map and it is unique.
\end{theorem}
\begin{proof}
Let us first prove that $\gamma$ is induced by a map. The uniqueness assertion will follows immediately by a standard argument. By Lemma \ref{Art3:2} and Remark \ref{Art3:3}, there exists an $m$-tuple $(u_1, \ldots, u_m)$ of $c$-splitting functions for $spt(\gamma)$  satisfying \eqref{Art3:1}. Fix $i\in \{0,1, \ldots, r\}$ and set $k_{0}=1$. From \eqref{Art3:1}, we deduce that the function $u_{k_i}(x_{k_i})$ is semi-concave for each $k_i$, as it is the infimum of semi-concave functions. Hence, $u_{k_i}(x_{k_i})$ is differentiable almost everywhere with respect to local coordinates. It follows that $u_{k_i}(x_{k_i})$ is differentiable $\mu_{k_i}$ almost everywhere, as the measure $\mu_{k_i}$ is absolutely continuous. It implies that $\gamma ( S)=1$, where
\begin{align*}
S&:=\Big\{ (x_1,x_{2}, \ldots, x_{m})\in \prod_{i=1}^{m}X_i:\text{$Du_{1}(x_1)$ \;and} \;\;\text{$Du_{k_{i}}(x_{k_{i}})$ exist for each $i=1, \ldots r$}, \;\text{and}\;\\
&\qquad \qquad \qquad\qquad \qquad\qquad \quad\;\;\sum_{i=1}^{m}u_i(x_i)=c(x_1, x_2, \ldots,x_m)\Big\}.
\end{align*}
 
Fix $x_1^{0} \in \pi_{1}(S)$. Clearly, $M_{x_1^{0}k_1\ldots k_r}\neq \emptyset$, and so by Lemma \ref{Art3:10} the map $(x_{2}, \ldots, x_{m})\mapsto D_{x_{1}}c(x_{1}^{0},x_2,\ldots, x_{m}) $ is injective on
the subset of  $M_{x_1^{0}k_1\ldots k_r}$ where $D_{x_1}c(x_1^{0}, x_2, \ldots, x_m)$ exists, this happens if and only if the set $M_{x_1^{0}k_1\ldots k_r}$  is a singleton (see Remark \ref{Art3:58}), which implies $W_{x_1^{0}k_1\ldots k_r}$ is also a singleton. This completes the proof that $\gamma$ is induced by a map. To prove that $\gamma$ is unique note that for any pair of solutions $\gamma_{1}$ and $\gamma_{2}$ (which are induced by maps $T_{1}$ and $T_{2}$), we have $\frac{1}{2}\left( \gamma_{1} + \gamma_{2}\right)$ is also a solution (by the convexity of the set $\Pi(\mu_{1},\ldots , \mu_{m})$), which implies that it is also concentrated on the graph of some map. However, $\frac{1}{2}\left( \gamma_{1} + \gamma_{2}\right)$ must be concentrated on the union of the graphs of $T_{1}$ and $T_{2}$. We conclude $T_{1}=T_{2}$ $\mu_1$-a.e., completing the proof of the theorem.
\end{proof}
\begin{remark}
Note from the above proof that the regularity condition on the first marginal (which is a standard assumption in the classical and multi-marginal optimal transport for uniqueness results), allow us to focus on the set $\{x_1\in X_1: Du_1(x_1) \;\; \text{exists}\}$, for every $m$-tuple $(u_1, \ldots, u_m)$ of Borel functions satisfying inequality \eqref{Art3:6}. In what follows such regularity condition holds, so to get uniqueness of solutions in the Monge-Kantorovich problem it suffices to prove that  the set $M_{x_1^{0}k_1\ldots k_r}(u_1, \ldots,u_m)$ is a singleton for every $x_1^{0}\in\{x_1\in X_1: Du_1(x_1) \;\; \text{exists}\}$ fixed, for every $m$-tuple $(u_1, \ldots, u_m)$ of Borel functions satisfying inequality \eqref{Art3:6} (see also Remark \ref{Art3:58}).
\end{remark}
\section{\textbf{Examples.
 }}\label{sect: examples}
 Here, we illustrate the result obtained in Theorem \ref{Main Theorem} throughout several examples.\\
For the next proposition, let us recall some basic concepts from graph theory. An \textit{undirected simple graph} $G$ on the set $\{1, \ldots,m\}$ is an ordered pair $(V(G), E(G))$, consisting of a  \textit{set of vertices} $V(G)=\{1, \ldots,m\}$ and a \textit{set of edges} $E(G)\subseteq\left\lbrace \{i,j\}:\text{ $i,j \in V(G)$ and $i\neq j$} \right\rbrace$. Given $i,j\in V(G)$, $i$ and $j$ are called \textit{adjacent} if $\{i,j\}\in E(G)$.  A \textit{path} is a nonempty sequence $\{\{i_1, i_2\},\ldots, \{i_{l-1}, i_l\}\}\subseteq E(G)$ with $i_k\neq i_r$ for all $k\neq r$. 
\begin{proposition}[One dimensional sub-modular type costs]
Assume $c(x_1,\ldots, x_m)$ is semi-concave and $C^{2}$, where $X_i=\mathbb{R}$ for all $i=1,\ldots, m$. Let $G$ be an undirected simple graph on $\{1, \ldots, m\}$ and assume
\begin{enumerate}
\item $\dfrac{\partial^{2}c}{\partial x_{i}\partial x_j}\leq 0$ for all $i\neq j$ and $\dfrac{\partial^{2}c}{\partial x_{i}\partial x_j}< 0$ for all $\{i,j\} \in E(G)$.
\item There  exists a set $P:=\{k_1,\ldots, k_r\}\subseteq \{1,\ldots, m\}$   such that for every $i\in \{1,\ldots, m\}$ not adjacent to  $1$,  there is a path $\{\{1, i_1\},\{i_1, i_2\},\ldots, \{i_{l-1}, i_l\}, \{i_l, i\}\}$ in $G$ with $\{i_1, \ldots, i_l\}\subseteq P$.
\end{enumerate}
Then $c$ is twisted on $c$-splitting sets with respect to the variables $x_1, x_{k_1}, \ldots, x_{k_r}$.
\end{proposition}
\begin{proof}
Let $(u_1, \ldots, u_m)$ be an $m$-tuple of Borel functions satisfying inequality \eqref{Art3:6} and fix $x_{1}^{0}\in X_1$ such that $Du_1(x_{1}^{0})$ exists and $M_{x_{1}^{0}k_1\ldots k_r}\neq \emptyset $. We want to prove that $M_{x_{1}^{0}k_1\ldots k_r}$ is a singleton. This will complete the proof.\\
 Let $(x_2, \ldots, x_m), (\overline{x}_2, \ldots, \overline{x}_m)\in M_{x_{1}^{0}k_1\ldots k_r}$ and set $x=(x_{1}^{0}, x_2, \ldots, x_m)$ and $\overline{x}=(x_{1}^{0}, \overline{x}_2, \ldots, \overline{x}_m)$. Consider 
 $$x^{+}=(x_{1}^{0}, x_2^{+}, \ldots, x_m^{+})\quad \text{where}\quad x_k^{+}=\text{max}\{x_k, \overline{x}_k\},$$
  $$x^{-}=(x_{1}^{0}, x_2^{-}, \ldots, x_m^{-})\quad \text{where}\quad x_k^{-}=\text{min}\{x_k, \overline{x}_k\}.$$
From definition of $M_{x_{1}^{0}k_1\ldots k_r}$ the set $\{x, \overline{x}\}$ is a $c$-splitting set, so it is cyclically  monotone. Then
\begin{equation}\label{Art3:50}
c(x)+ c(\overline{x}) \leq c(x^{+}) + c(x^{-}).
\end{equation}
We claim that the reverse inequality also holds. To get this consider $x(t)=tx^{+} + (1-t)x$ and $y(t)=t\overline{x} + (1-t)x^{-}$ for $s\in [0,1]$. Next, write
\begin{align}\label{Art3:51}
c(x^{+}) - c(x) &= \int_{0}^{1}\dfrac{d}{dt}c(x(t))dt \nonumber\\
&=\int_{0}^{1}\sum_{i=2}^{m}\dfrac{\partial c(x(t))}{\partial x_i}(x_{i}^{+}-x_i)dt, \nonumber\\
\end{align}
and 
\begin{align}\label{Art3:52}
c(x^{-}) - c(\overline{x}) &= -\int_{0}^{1}\dfrac{d}{dt}c(y(t))dt \nonumber\\
&=-\int_{0}^{1}\sum_{i=2}^{m}\dfrac{\partial c(y(t))}{\partial x_i}(\overline{x}_i-x_i^{-})dt. \nonumber\\
\end{align}
Since for each $i\in\{2,\ldots, m\}$ we have
\begin{equation}\label{Art3:54}
x_{i}^{+}-x_i=\overline{x}_i-x_i^{-}= \begin{cases} 
      \overline{x}_{i}-x_i & \text{if}\quad \overline{x}_{i}>x_i \\
       0 & \overline{x}_{i}\leq x_i, \\
        \end{cases}
\end{equation}
the addition of \eqref{Art3:51} and \eqref{Art3:52} gives
\begin{equation}\label{Art3:55}
c(x^{+}) - c(x) + c(x^{-}) - c(\overline{x})= \int_{0}^{1}\sum_{i=2}^{m}\left[\dfrac{\partial c(x(t))}{\partial x_i}- \dfrac{\partial c(y(t))}{\partial x_i}\right](x_{i}^{+}-x_i)dt
\end{equation}
Now,  set $x(t,s)=sx(t) + (1-s)y(t)$, with $t\in [0,1]$ fixed. Then for each $i\in\{2,\ldots, m\}$ we have
\begin{align*}\label{Art3:53}
\dfrac{\partial c(x(t))}{\partial x_i}- \dfrac{\partial c(y(t))}{\partial x_i} &= \int_{0}^{1}\sum_{j=2}^{m}\dfrac{\partial^{2} c(x(t,s))}{\partial x_i \partial x_j}\left[tx_{j}^{+} + (1-t)x_j -\right( t\overline{x}_j + (1-t)x_{j}^{-}\left)\right]ds \nonumber\\
&=\int_{0}^{1}\sum_{j=2}^{m}\dfrac{\partial^{2} c(x(t,s))}{\partial x_i \partial x_j}\left[t(x_{j}^{+} -x_j) -t(\overline{x}_j - x_{j}^{-}) + x_j - x_j^{-}\right]ds \nonumber\\
&=\int_{0}^{1}\sum_{j=2}^{m}\dfrac{\partial^{2} c(x(t,s))}{\partial x_i \partial x_j}\left( x_j - x_j^{-}\right)ds. \qquad\qquad\qquad\quad \text{by \eqref{Art3:54}}\nonumber\\
\end{align*}
Substituting it into \eqref{Art3:55} we get 
\begin{equation*}
c(x^{+}) - c(x) + c(x^{-}) - c(\overline{x})=\int_{0}^{1}\int_{0}^{1}\sum_{i,j=2}^{m}\dfrac{\partial^{2} c(x(t,s))}{\partial x_i \partial x_j}\left( x_i^{+} - x_i\right)\left( x_j - x_j^{-}\right)dsdt.
\end{equation*}
Now note that $x_i^{+} - x_i, x_j - x_j^{-}\geq 0$, then by Assumption 1, $c(x^{+}) - c(x) + c(x^{-}) - c(\overline{x})\leq 0$. This implies that equality holds in \eqref{Art3:50}, completing the proof of the claim. Also, note that if one of the inequalities 
\begin{equation}\label{Art3:61}
u_{1}(x_1^{0})+\sum_{i=2}^{m}u_{i}(x_{i}^{+})\leq c(x^{+})
\end{equation}
\begin{equation}\label{Art3:62}
u_{1}(x_1^{0})+\sum_{i=2}^{m}u_{i}(x_{i}^{-})\leq c(x^{-})
\end{equation}
is strict, we would have
\begin{align}
2u_{1}(x_1^{0}) + \sum_{i=2}^{m}u_{i}(x_{i}^{+}) + \sum_{i=2}^{m}u_{i}(x_{i}^{-}) &< c(x^{+}) + c(x^{-}) \nonumber\\
&= c(x) +  c(\overline{x}) \nonumber\\
&=2u_{1}(x_1^{0}) + \sum_{i=2}^{m}u_{i}(x_{i}) + \sum_{i=2}^{m}u_{i}(\overline{x_{i}})\nonumber\\
&=2u_{1}(x_1^{0}) + \sum_{i=2}^{m}u_{i}(x_{i}^{+}) + \sum_{i=2}^{m}u_{i}(x_{i}^{-}),\nonumber\\
\end{align}
which is clearly not posible; that is, equality holds in \eqref{Art3:61} and \eqref{Art3:62}. Hence
\begin{equation}\label{Art3:56}
x^{+}, x^{-}\in M_{x_{1}^{0}k_1\ldots k_r}. 
\end{equation}

Furthermore, from Lemma \ref{Art3:8} we get
$$\dfrac{\partial c(x^{+})}{\partial x_1}=Du_{1}(x_{1}^{0})=\dfrac{\partial c(x^{-})}{\partial x_1},$$
or equivalently,
$$\int_{0}^{1}\sum_{i=2}^{m}\dfrac{\partial^{2} c(r(t))}{\partial x_1\partial x_i}(x_{i}^{+}-x_i^{-})dt=0,$$
where $r(t)= tx^{+} + (1-t)x^{-}$, $t\in [0,1]$. We then must have $$\dfrac{\partial^{2} c(r(t))}{\partial x_1\partial x_i}(x_{i}^{+}-x_i^{-})=0$$ for every $i\in \{2,\ldots, m\}$, as $\dfrac{\partial^{2} c(r(t))}{\partial x_1\partial x_i}(x_{i}^{+}-x_i^{-})\leq 0$ on $\{2,\ldots, m\}$. We next use Assumption 1 to deduce $x_{i}^{+}=x_{i}^{-}$ for all $i$ adjacent to $1$; that is,
\begin{equation}\label{Art3:57}
x_{i}=\overline{x}_i\;\;\text{for all  $i$ adjacent to $1$.}
\end{equation}
Now, if $1$ is adjacent to all the other vertices, the proof is completed. If there is a vertex not adjacent to $1$, then $1$ must be adjacent to some $i\in P$ (by Assumption 2), which implies $x_{i}=\overline{x}_i$ by \eqref{Art3:57}. Combining this with \eqref{Art3:56} and Lemma \ref{Art3:8} we get
$$\dfrac{\partial c(x^{+})}{\partial x_i}=Du_{i}(x_{i})=\dfrac{\partial c(x^{-})}{\partial x_i},$$ so we can mimic the arguments presented in the proof of \eqref{Art3:57} (beginning from \eqref{Art3:56}) to get $x_{j}=\overline{x}_j$ for every $j$ adjacent to $i$. Following this iterative process we can prove that $x_{j}=\overline{x}_j$ for every $j\in V(G)$, as Assumption 2 implies that every vertex of $V(G)$  is adjacent to at least one vertex in $P$, completing the proof of the proposition.
\end{proof}
\begin{remark}
Note that  if the graph $G$ is complete, we can take $P=\emptyset$ and Condition 1 basically means that $c$ is strictly sub-modular. Unique Monge type solutions for strictly sub-modular costs was established by Carlier \cite{C2010}.  It was observed in \cite{Pass0} that this condition is equivalent (up to a change of variables) to the compatibility condition, which states that $$\left(\dfrac{\partial^{2}c}{\partial x_{i}\partial x_j}\right)\left(\dfrac{\partial^{2}c}{\partial x_{k}\partial x_j}\right)^{-1}\left(\dfrac{\partial^{2}c}{\partial x_{k}\partial x_i}\right)<0$$ everywhere,  for all distinct $i,j,k$, and so compatible costs yield unique Monge solutions as well. 
\end{remark}

We can easily see that the next result is a generalization of a special case of Theorem 3.1 in \cite{PV2021}. Note that here we do not require $f$ being symmetric.
\par

\begin{proposition}\label{Art3:39}
Let $\{I_1, I_2, I_3\}$ be a partition of $\{1, \ldots, m\}$. Let $f:\mathbb{R}^{n}\times \mathbb{R}^{n}\mapsto \mathbb{R}$ be a function satisfying:
 \begin{enumerate}
 \item \label{Art3:36} $f$ is bi-linear,
 \item $f(x,x)\leq 0$ for every $x\in \mathbb{R}^{n}$,
 \item \label{Art3:46} $f$ is bi-twisted; that is, for each $x_0, y_0 \in \mathbb{R}^{n}$ fixed,  the maps $y\mapsto D_x f(x_0,y)$ and $x\mapsto D_y f(x,y_0)$ are injective on $\{x_0\}\times \mathbb{R}^{n}$ and $\mathbb{R}^{n} \times \{y_0\}$ respectively.
 \end{enumerate}

Assume $1\in I_1$ and fix $p\in I_2\cup I_3$, then the cost function 
\begin{equation}\label{Art3:26}
c(x_1, \ldots, x_m) = \sum_{s\in I_1}\sum_{t\in I_2\cup I_3}f(x_{s},x_{t}) +  \sum_{s\in I_3}\sum_{t\in I_2}f(x_{s},x_{t}) + \sum_{\underset{s<t }{s,t\in  I_3}}f(x_{s},x_{t})
\end{equation}
is twisted on $c$-splitting sets with respect to $x_1$ and $x_p$.
\end{proposition}
\begin{proof}
Firstly, by hypothesis 1 we can write 
\begin{flalign}\label{Art3:37}
c(x_1, \ldots, x_m)&=f(\sum_{s\in I_1}x_{s},\sum_{t\in I_2\cup I_3}x_{t}) + f(\sum_{s\in I_3}x_{s},\sum_{t\in I_2}x_{t}) + \sum_{\underset{s<t }{s,t\in  I_3}}f(x_{s},x_{t}).
\end{flalign}
Let $(u_1, \ldots, u_m)$ be an $m$-tuple of Borel functions satisfying inequality \eqref{Art3:6} and fix $x_1^{0}\in \{x_1\in X_1:Du_1(x_1)\;\; \text{exists}\}$, with $M_{x_1^{0}p}(u_1, \ldots, u_m)\neq \emptyset$. We want to prove that $M_{x_1^{0}p}$ is a singleton, this will complete the proof.\\
Let $(x_2^{1}, \ldots, x_m^{1}), (x_2^{2}, \ldots, x_m^{2})\in M_{x_1^{0}p}$. Since $\{I_1,I_2,I_3\}$ is a partition and $1\in I_1$, we get from  Lemma \ref{Art3:8} and \eqref{Art3:37},
\begin{flalign*}
D_{x_{1}}f(x_1^{0}, \sum_{t\in I_2\cup I_3}x_t^{1})& =D_{x_1}c(x_1^{0},x_2^{1}, \ldots, x_m^{1})\nonumber\\
&=Du_{1}(x_{1}^{0})\qquad\qquad\qquad\qquad \nonumber\\
&=D_{x_1}c(x_1^{0},x_2^{2}, \ldots, x_m^{2})\nonumber\\
  &=D_{x_{1}}f(x_1^{0}, \sum_{t\in I_2\cup I_3}x_t^{2}).\nonumber\\
\end{flalign*}
It follows that 
\begin{equation}\label{Art3:23}
\sum_{t\in I_2\cup I_3}x_t^{1}= \sum_{t\in I_2\cup I_3}x_t^{2},
\end{equation}
by Assumption \ref{Art3:46}.
\begin{claim}
For every $N\subseteq I_1$ we get $y_{N}:=(y_2, \ldots, y_m)\in M_{x_1^{0}p}$, where 
\begin{equation}\label{Art3:34}
y_{s}= \begin{cases} 
      x_{s}^{2} & \text{if}\quad s\in \{2, \ldots, m\}\setminus N \\
      x_{s}^{1} & \text{if}\quad  s\in N. \\

   \end{cases}
   \end{equation}
\end{claim}
\par
\begin{flushleft}
\textit{Proof of Claim 1.} Note that from \eqref{Art3:37}, we can write 
\end{flushleft}
\begin{flalign}
c(x_1, \ldots, x_m)&=f(\sum_{s\in N}x_{s},\sum_{t\in I_2\cup I_3}x_{t}) + f(\sum_{s\in I_1\setminus N}x_{s},\sum_{t\in I_2\cup I_3}x_{t})  + f(\sum_{s\in I_3}x_{s},\sum_{t\in I_2}x_{t}) + \sum_{\underset{s<t }{s,t\in  I_3}}f(x_{s},x_{t}).
\end{flalign}
Since $(x_{2}^{1}, \ldots, x_m^{1})\in M_{x_1^{0}p}$, we get
\begin{align*}
\left\lbrace x_{s}^{1}\right\rbrace_{s\in N}&\in\text{Argmin}\Bigg\{\left\lbrace x_{s}\right\rbrace_{s\in N} \mapsto f(\sum_{s\in N}x_{s},\sum_{t\in I_2\cup I_3}x_{t}^{1})-\sum_{s\in N}u_{s}(x_{s}) + f(\sum_{s\in I_1\setminus N}x_{s}^{1},\sum_{t\in I_2\cup I_3}x_{t}^{1}) \\ &
\qquad\qquad\qquad   + f(\sum_{s\in I_3}x_{s}^{1},\sum_{t\in I_2}x_{t}^{1}) + \sum_{\underset{s<t }{s,t\in  I_3}}f(x_{s}^{1},x_{t}^{1}) -\sum_{s\in \{1, \ldots, m\} \setminus N}u_{s}(x_{s}^{1})\Bigg\}\\
&=\text{Argmin}\Bigg\{\left\lbrace x_{s}\right\rbrace_{s\in N} \mapsto f(\sum_{s\in N}x_{s},\sum_{t\in I_2\cup I_3}x_{t}^{1})-\sum_{s\in N}u_{s}(x_{s}) \Bigg\}\\
&=\text{Argmin}\Bigg\{\left\lbrace x_{s}\right\rbrace_{s\in N} \mapsto f(\sum_{s\in N}x_{s},\sum_{t\in I_2\cup I_3}x_{t}^{2})-\sum_{s\in N}u_{s}(x_{s}) \Bigg\},
\end{align*}
by \eqref{Art3:23}. We deduce $y_N\in M_{x_1^{0}p}$, as $(x_2^{2}, \ldots, x_{m}^{2})\in M_{x_1^{0}p}$. This complete the proof of Claim 1.
\begin{claim}
$x_s^{1}=x_s^{2}$ for every $s\in I_2$.
\end{claim}
\begin{flushleft}
\textit{Proof of Claim 2.} From Claim 1, $(y_2, \ldots,y_m)\in M_{x_1^{0}p}$ where 
\end{flushleft}
\begin{equation}\label{Art3:35}
y_{s}= \begin{cases} 
      x_{s}^{2} & \text{if}\quad s\in \{2, \ldots, m\}\setminus I_1 \\
      x_{s}^{1} & \text{if}\quad  s\in I_1. \\
   \end{cases}
   \end{equation}
 Then, by fixing $r\in I_2$ we get
\begin{flalign*}
x_{r}^{1}&\in\text{Argmin}\Bigg\{x_{r} \mapsto c(x_1^{0}, x_2^{1}, \ldots, x_{r-1}^{1},x_{r},x_{r+1}^{1}, \ldots, x_m^{1}) - u_{r}(x_{r})\Bigg\},
\end{flalign*} 
\begin{flalign*}
y_{r}=x_{r}^{2}&\in\text{Argmin}\Bigg\{x_{r} \mapsto c(x_1^{0}, y_2, \ldots, y_{r-1},x_{r},y_{r+1}, \ldots, y_m) - u_{r}(x_{r})\Bigg\}.
\end{flalign*}
Then
\begin{equation}\label{Art3:24}
c(x_1^{0}, x_2^{1}, \ldots, x_m^{1}) - u_{r}(x_{r}^{1})\leq c(x_1^{0}, x_2^{1}, \ldots, x_{r-1}^{1},x_{r}^{2},x_{r+1}^{1}, \ldots, x_m^{1}) - u_{r}(x_{r}^{2}),
\end{equation}
\begin{equation}\label{Art3:25}
c(x_1^{0}, y_2, \ldots, y_m) - u_{r}(x_{r}^{2})\leq c(x_1^{0}, y_2, \ldots, y_{r-1},x_{r}^{1},y_{r+1}, \ldots, y_m) - u_{r}(x_{r}^{1}),
\end{equation}
which implies 
\begin{flalign}\label{Art3:22}
c(x_1^{0}, x_2^{1}, \ldots, x_m^{1}) + c(x_1^{0}, y_2, \ldots, y_m) &\leq c(x_1^{0}, x_2^{1}, \ldots, x_{r-1}^{1},x_{r}^{2},x_{r+1}^{1}, \ldots, x_m^{1})\nonumber \\
&+  c(x_1^{0}, y_2, \ldots, y_{r-1},x_{r}^{1},y_{r+1}, \ldots, y_m).
\end{flalign}
Now, from bi-linearity of $f$ we can write   
\begin{flalign*}
c(x_1, \ldots, x_m)&=f(\sum_{s\in I_1}x_{s},\sum_{t\in I_2\cup I_3}x_{t}) + f(\sum_{s\in I_3}x_{s},\sum_{t\in I_2}x_{t}) + \sum_{\underset{s<t }{s,t\in  I_3}}f(x_{s},x_{t})\nonumber\\
&= f(\sum_{s\in I_1\cup I_3}x_{s},\sum_{t\in I_2}x_{t}) +  f(\sum_{s\in I_1}x_{s},\sum_{t\in I_3}x_{t}) + \sum_{\underset{s<t }{s,t\in  I_3}}f(x_{s},x_{t})\nonumber\\
&=f(\sum_{s\in I_1\cup I_3}x_{s},x_r) + f(\sum_{s\in I_1\cup I_3}x_{s},\sum_{t\in I_2\setminus \{r\} }x_{t}) +  f(\sum_{s\in I_1}x_{s},\sum_{t\in I_3}x_{t}) + \sum_{\underset{s<t }{s,t\in  I_3}}f(x_{s},x_{t}).
\end{flalign*}
Combining this with \eqref{Art3:22} we get
\begin{flalign*}
f(\sum_{s\in I_1\cup I_3}x_{s}^1,x_r^1) + f(\sum_{s\in I_1\cup I_3}x_{s}^1,\sum_{t\in I_2\setminus \{r\} }x_{t}^1) +  f(\sum_{s\in I_1}x_{s}^1,\sum_{t\in I_3}x_{t}^1) + \sum_{\underset{s<t }{s,t\in  I_3}}f(x_{s}^1,x_{t}^1)\\
+ f(\sum_{s\in I_1\cup I_3}y_{s},y_r) + f(\sum_{s\in I_1\cup I_3}y_{s},\sum_{t\in I_2\setminus \{r\} }y_{t}) +  f(\sum_{s\in I_1}y_{s},\sum_{t\in I_3}y_{t}) + \sum_{\underset{s<t }{s,t\in  I_3}}f(y_{s},y_{t})\\
\leq f(\sum_{s\in I_1\cup I_3}x_{s}^1,x_r^2) + f(\sum_{s\in I_1\cup I_3}x_{s}^1,\sum_{t\in I_2\setminus \{r\} }x_{t}^1) +  f(\sum_{s\in I_1}x_{s}^1,\sum_{t\in I_3}x_{t}^1) + \sum_{\underset{s<t }{s,t\in  I_3}}f(x_{s}^1,x_{t}^1)\\
 + f(\sum_{s\in I_1\cup I_3}y_{s},x_r^{1}) + f(\sum_{s\in I_1\cup I_3}y_{s},\sum_{t\in I_2\setminus \{r\} }y_{t}) +  f(\sum_{s\in I_1}y_{s},\sum_{t\in I_3}y_{t}) + \sum_{\underset{s<t }{s,t\in  I_3}}f(y_{s},y_{t}).
\end{flalign*}
Then, the above inequality reduces to 
\begin{flalign}\label{Art3:26}
f(\sum_{s\in I_1\cup I_3}x_{s}^1,x_r^1) + f(\sum_{s\in I_1\cup I_3}y_{s},y_r)
\leq f(\sum_{s\in I_1\cup I_3}x_{s}^1,x_r^2) + f(\sum_{s\in I_1\cup I_3}y_{s},x_r^{1}).
\end{flalign}
By construction of $y$ and linearity we have
$$f(\sum_{s\in I_1\cup I_3}y_{s},y_r)=f(\sum_{s\in I_1}y_{s},y_r)+f(\sum_{s\in  I_3}y_{s},y_r)=f(\sum_{s\in I_1}x_{s}^1,x_r^{2})+f(\sum_{s\in  I_3}x_{s}^{2},x_r^{2}),$$
$$f(\sum_{s\in I_1\cup I_3}y_{s},x_r^{1})=f(\sum_{s\in I_1}y_{s},x_r^{1})+f(\sum_{s\in  I_3}y_{s},x_r^{1})=f(\sum_{s\in I_1}x_{s}^1,x_r^{1})+f(\sum_{s\in  I_3}x_{s}^{2},x_r^{1}).$$
Substituting it into \eqref{Art3:26} and eliminating similar terms we get
$$f(\sum_{s\in I_3}x_{s}^1,x_r^1) + f(\sum_{s\in I_3}x_{s}^{2},x_r^{2})
\leq f(\sum_{s\in I_3}x_{s}^1,x_r^2) + f(\sum_{s\in I_3}x_{s}^{2},x_r^{1}),$$
and by \eqref{Art3:23}, we get
$$f(\sum_{s\in I_2}(x_{s}^2-x_{s}^1),x_r^1-x_r^2)=f(\sum_{s\in I_3}(x_{s}^1-x_{s}^2),x_r^1-x_r^2)\leq 0;$$
 that is,  
 $$f(\sum_{s\in I_2}(x_{s}^2-x_{s}^1),x_r^2-x_r^1)\geq 0.$$
 Summing over $r\in I_2$ we get
 $$f(\sum_{s\in I_2}(x_{s}^2-x_{s}^1),\sum_{r\in I_2}(x_{r}^2-x_{r}^1))\geq 0.$$
 Combining this with hypothesis  2 we get $f(\sum_{s\in I_2}(x_{s}^2-x_{s}^1),\sum_{r\in I_2}(x_{r}^2-x_{r}^1))= 0$. Then, we must have, in particular, equality in \eqref{Art3:24}. It follows that $(x_2^{1}, \ldots, x_{r-1}^{1},x_{r}^{2},x_{r+1}^{1}, \ldots, x_m^{1})\in M_{x_1^{0}p}$, and so
\begin{flalign*}
D_{x_{1}}f(x_1^{0}, \sum_{t\in I_2\cup I_3}x_t^{1})& =D_{x_1}c(x_1^{0},x_2^{1}, \ldots, x_m^{1})\nonumber\\
&=Du_{1}(x_{1}^{0})  \nonumber\\
&=D_{x_1}c(x_1^{0}, x_2^{1}, \ldots, x_{r-1}^{1},x_{r}^{2},x_{r+1}^{1}, \ldots, x_m^{1})\nonumber\\
  &=D_{x_{1}}f(x_1^{0}, x_{r}^{2} + \sum_{t\in I_2\cup I_3\setminus \{r\}}x_t^{1}).\nonumber\\
\end{flalign*}
Thus, $x_{r}^{1}=x_{r}^{2}$, as $f$ is twisted. This completes the proof of Claim 2.
\begin{claim}
For every $n$, equation \eqref{Art3:23} implies $x_{t_j}^{1}=x_{t_j}^{2}$ for $1\leq j\leq n$, where $I_3:= \{t_1, \ldots, t_n\}$.
\end{claim}
\textit{Proof of Claim 3.}
From Claim 2 and \eqref{Art3:23}  we get
\begin{equation}\label{Art3:47}
\sum_{t\in I_3}x_t^{1}= \sum_{t\in  I_3}x_t^{2}.
\end{equation}
We proceed to apply induction on $n$. Indeed, when $n=1$ it is clearly true. Assume the statement is true when $n=k-1$. Note that
$$f(\sum_{s\in I_1}x_{s},\sum_{t\in I_3}x_{t})=f(\sum_{s\in I_1}x_{s},x_{t_k})+ f(\sum_{s\in I_1}x_{s},\sum_{t\in I_3\setminus \{t_k\}}x_{t}),$$
$$f(\sum_{s\in I_3}x_{s},\sum_{t\in I_2}x_{t})= f(x_{t_k},\sum_{t\in I_2}x_{t})+f(\sum_{s\in I_3\setminus \{t_k\}}x_{s},\sum_{t\in I_2}x_{t}),$$
$$\sum_{\underset{s<t }{s,t\in  I_3}}f(x_{s},x_{t})=\sum_{s\in  I_3\setminus \{t_k\}}f(x_{s},x_{t_k}) + \sum_{\underset{s<t }{s,t\in  I_3\setminus \{t_k\}}}f(x_{s},x_{t}).$$
Hence,
\begin{flalign*}
c(x_1, \ldots, x_m)& =f(\sum_{s\in I_1}x_{s},\sum_{t\in I_2\cup I_3}x_{t}) + f(\sum_{s\in I_3}x_{s},\sum_{t\in I_2}x_{t}) + \sum_{\underset{s<t }{s,t\in  I_3}}f(x_{s},x_{t})\\
&=f(\sum_{s\in I_1}x_{s},\sum_{t\in I_2}x_{t}) + f(\sum_{s\in I_1}x_{s},\sum_{t\in I_3}x_{t}) + f(\sum_{s\in I_3}x_{s},\sum_{t\in I_2}x_{t}) + \sum_{\underset{s<t }{s,t\in  I_3}}f(x_{s},x_{t})\nonumber\\
&=f(\sum_{s\in I_1}x_{s},\sum_{t\in I_2}x_{t}) +f(\sum_{s\in I_1}x_{s},x_{t_k})+ f(\sum_{s\in I_1}x_{s},\sum_{t\in I_3\setminus \{t_k\}}x_{t})+f(x_{t_k},\sum_{t\in I_2}x_{t}) \nonumber\\
&+ f(\sum_{s\in I_3\setminus \{t_k\}}x_{s},\sum_{t\in I_2}x_{t}) + \sum_{s\in  I_3\setminus \{t_k\}}f(x_{s},x_{t_k}) + \sum_{\underset{s<t }{s,t\in  I_3\setminus \{t_k\}}}f(x_{s},x_{t}).\nonumber\\
\end{flalign*}
Since the only terms of $c$ depending on $x_{t_k}$ are $f(\sum_{s\in I_1}x_{s},x_{t_k}), f(x_{t_k},\sum_{t\in I_2}x_{t}) $ and $\sum_{s\in  I_3\setminus \{t_k\}}f(x_{s},x_{t_k})$, we get
\begin{flalign*}
x_{t_k}^{1}&\in\text{Argmin}\Bigg\{x_{t_k} \mapsto f(\sum_{s\in I_1}x_{s}^{1},x_{t_k})+ f(x_{t_k},\sum_{t\in I_2}x_{t}^{1}) + \sum_{s\in  I_3\setminus \{t_k\}}f(x_{s}^{1},x_{t_k})- u_{t_k}(x_{t_k})\Bigg\}.
\end{flalign*} 
Furthermore, defining $y$ as in \eqref{Art3:35} we get
\begin{flalign*}
y_{t_k}=x_{t_k}^{2}&\in\text{Argmin}\Bigg\{x_{t_k} \mapsto f(\sum_{s\in I_1}y_{s},x_{t_k})+ f(x_{t_k},\sum_{t\in I_2}y_{t}) + \sum_{s\in  I_3\setminus \{t_k\}}f(y_{s},x_{t_k})- u_{t_k}(x_{t_k})\Bigg\}.
\end{flalign*} 
We deduce 
\begin{flalign}\label{Art3:48}
f(\sum_{s\in I_1}x_{s}^{1},x_{t_k}^{1})+ f(x_{t_k}^{1},\sum_{t\in I_2}x_{t}^{1}) + \sum_{s\in  I_3\setminus \{t_k\}}f(x_{s}^{1},x_{t_k}^{1})- u_{t_k}(x_{t_k}^{1})\nonumber\\
 \leq f(\sum_{s\in I_1}x_{s}^{1},x_{t_k}^{2})+ f(x_{t_k}^{2},\sum_{t\in I_2}x_{t}^{1}) + \sum_{s\in  I_3\setminus \{t_k\}}f(x_{s}^{1},x_{t_k}^{2})- u_{t_k}(x_{t_k}^{2}),
\end{flalign}
\begin{flalign*}
f(\sum_{s\in I_1}y_{s},x_{t_k}^{2})+ f(x_{t_k}^{2},\sum_{t\in I_2}y_{t}) + \sum_{s\in  I_3\setminus \{t_k\}}f(y_{s},x_{t_k}^{2})- u_{t_k}(x_{t_k}^{2})\\
 \leq f(\sum_{s\in I_1}y_{s},x_{t_k}^{1})+ f(x_{t_k}^{1},\sum_{t\in I_2}y_{t}) + \sum_{s\in  I_3\setminus \{t_k\}}f(y_{s},x_{t_k}^{1})- u_{t_k}(x_{t_k}^{1}).
\end{flalign*}
Adding the above inequalities,  using Claim 2 and construction of $y$ we get
$$\sum_{s\in  I_3\setminus \{t_k\}}f(x_{s}^{1},x_{t_k}^{1}) +\sum_{s\in  I_3\setminus \{t_k\}}f(x_{s}^{2},x_{t_k}^{2})\leq \sum_{s\in  I_3\setminus \{t_k\}}f(x_{s}^{1},x_{t_k}^{2}) +\sum_{s\in  I_3\setminus \{t_k\}}f(x_{s}^{2},x_{t_k}^{1}).$$
Combining this with \eqref{Art3:47} we get
 $$f(x_{t_k}^{2}-x_{t_k}^{1},x_{t_k}^{2}-x_{t_k}^{1})=f(\sum_{s\in I_3\setminus \{t_k\}}(x_{s}^1-x_{s}^2),x_{t_k}^2-x_{t_k}^1)\geq 0.$$
 From hypothesis 2 we then get $f(x_{t_k}^{2}-x_{t_k}^{1},x_{t_k}^{2}-x_{t_k}^{1})=0$. Hence, equality holds in \eqref{Art3:48} and $(x_2^{1}, \ldots, x_{t_{k}-1}^{1},x_{t_k}^{2},x_{t_k+1}^{1}, \ldots, x_m^{1})\in M_{x_1^{0}p}$. This implies
\begin{flalign*}
D_{x_{1}}f(x_1^{0}, \sum_{t\in I_2\cup I_3}x_t^{1})& =D_{x_1}c(x_1^{0},x_2^{1}, \ldots, x_m^{1})\nonumber\\
&=Du_{1}(x_{1}^{0})  \nonumber\\
&=D_{x_1}c(x_1^{0}, x_2^{1}, \ldots, x_{t_k-1}^{1},x_{t_k}^{2},x_{t_k+1}^{1}, \ldots, x_m^{1})\nonumber\\
  &=D_{x_{1}}f(x_1^{0}, x_{t_k}^{2} + \sum_{t\in I_2\cup I_3\setminus \{t_k\}}x_t^{1}).\nonumber\\
\end{flalign*}
Thus, $x_{t_k}^{1}=x_{t_k}^{2}$, as $f$ is twisted. Hence, from \eqref{Art3:47} and Claim 2 we can write $\sum_{t\in I_2\cup I_3\setminus \{t_k\}}x_t^{1}= \sum_{t\in I_2\cup I_3\setminus \{t_k\}}x_t^{2}$. It follows that $x_{t_2}^{1}=x_{t_2}^{2}, \ldots, x_{t_{k-1}}^{1}=x_{t_{k-1}}^{2}$, by induction hypothesis. This completes the proof of Claim 3. 
\begin{claim}
$x_{s}^{1}=x_{s}^{2}$ for every $s\in I_1$.
\end{claim}
\textit{Proof of Claim 3.} Since $p\in I_2\cup I_3$, $x_p^1=x_p^2$ by Claim 2 and 3. Hence,
\begin{flalign*}
 D_{x_p}c(x_1^{0},x_2^{1}, \ldots, x_m^{1})& =Du_{p}(x_{p}^{1})\nonumber\\
&=Du_{p}(x_{p}^{2})  \nonumber\\
&= D_{x_p}c(x_1^{0},x_2^{2}, \ldots, x_m^{2}).\nonumber\\
\end{flalign*}
Combining the above equality, Claim 2 and 3,  and  \eqref{Art3:37} we get  $D_{x_{p}}f(\sum_{t\in I_1}x_t^{1}, x_p^{1})=D_{x_{p}}f(\sum_{t\in I_1}x_t^{2},  x_p^{2})$. It follows that 

\begin{equation}\label{Art3:27}
\sum_{t\in I_1}x_t^{1}= \sum_{t\in I_1}x_t^{2}.
\end{equation}
Now, fix $t\in I_1\setminus \{1\}$. Setting $N=\{t\}$, we use Claim 1 to get $y_N=(x_2^{2}, \ldots, x_{t-1}^{2}, x_{t}^{1}, x_{t+1}^{2},\ldots, x_m^{2})\in M_{x_1^{0}p}$. Since \eqref{Art3:27} holds true for every $(x_2^{1}, \ldots, x_m^{1}), (x_2^{2}, \ldots, x_m^{2})\in M_{1p}$, in particular, it is true for $y_N$ and $(x_2^{2}, \ldots, x_m^{2})$. It immediately implies $x_t^{1}=x_t^{2}$, completing the proof of Claim 4.
\par
This completes the proof of the Proposition.
\end{proof}
The next result focuses on a cost with a cycle structure that generalizes the main result in \cite{PV2020}.
  \begin{proposition}\label{Art3:41}
Consider 
\begin{equation}\label{Aert3:38}
c(x_{1},x_{2},x_{3},x_{4})= c_{1}(x_{1},x_{2})+c_{2}(x_{2},x_{3})+c_{3}(x_{3},x_{4})+c_{4}(x_{4},x_{1}),
\end{equation}
with $c_{i}$ semi-concave for each $i=1,2,3,4$. 
Assume
\begin{enumerate}
\item \label{Ass:1} For every 4-tuple of Borel functions  $(u_{1},u_2,u_3, u_{4})$ satisfying inequality \eqref{Art3:6} and $x_1^{0} \in  \left\lbrace x_1\in X_1: Du_1(x_1)\; \text{exists}\right\rbrace$, we get 
\begin{equation}
c_{2}(x_{2}^{1},x_{3}^{1}) + c_{3}(x_{3}^{1},x_{4}^{1}) + c_{2}(x_{2}^{2},x_{3}^{2}) + c_{3}(x_{3}^{2},x_{4}^{2})\geq c_{2}(x_{2}^{1},x_{3}^{2}) + c_{3}(x_{3}^{2},x_{4}^{1}) + c_{2}(x_{2}^{2},x_{3}^{1}) + c_{3}(x_{3}^{1},x_{4}^{2}),
\end{equation}
 for every $ (x_{2}^{1}, x_{3}^{1},x_{4}^{1}), (x_{2}^{2},x_{3}^{2},x_{4}^{2})\in M_{x_1^{0}4}$.
\item \label{Ass:2} $c_3$ is twisted with respect to $x_4$; that is, for every $x_4$ fixed the map $x_{3}\mapsto D_{x_{4}}c_{3}(x_{3},x_{4})$ is injective on the subset of $X_3\times \{x_4\}$ where $c_3$ is differentiable with respect to $x_4$. 
\item \label{Ass:3}   $c_1$ and $c_4$ are twisted with respect to $x_1$ respectively; that is, for every $x_1$ fixed the maps $x_{2}\mapsto D_{x_{1}}c_{1}(x_{1},x_{2})$ and  $x_{4}\mapsto D_{x_{1}}c_{4}(x_{4},x_{1})$ are  injective on the subsets of $\{x_1\}\times X_2 $ and $ X_4\times \{x_1\}$ where $c_1$ and $c_4$ are differentiable with respect to $x_1$ respectively. 
\end{enumerate}
Then, $c$ is twisted on $c$-splitting sets with respect to $x_1$ and $x_4$.
\begin{proof}
 Let $(u_{1},u_2,u_3, u_{4})$ be a 4-tuple of Borel functions satisfying inequality \eqref{Art3:6}. Fix  $x_1^{0} \in  \left\lbrace x_1\in X_1: Du_1(x_1)\; \text{exists}\right\rbrace$ and let $(x_2^{1}, x_3^{1}, x_4^{1}), (x_2^{2}, x_3^{2}, x_4^{2})\in M_{x_1^{0}4}$. We want to show  $x_i^{1}=x_i^{2}$, $i=2,3,4$. For this, observe that
$$(x_{3}^{k},x_4^{k})\in  \text{Argmin}\Big\{(x_{3},x_4) \mapsto c(x_1^{0},x_2^{k},x_3,x_4)-u_{3}(x_{3})-u_{4}(x_{4})\Big\},\; k=1,2.$$ 
Then
\begin{equation}\label{Art3:4}
c(x_1^{0},x_2^{1},x_3^{1},x_4^{1})-u_{3}(x_{3}^{1})-u_{4}(x_{4}^{1})\leq c(x_1^{0},x_2^{1},x_3^{2},x_4^{1})-u_{3}(x_{3}^{2})-u_{4}(x_{4}^{1}),
\end{equation}
\begin{equation}\label{Art3:5}
c(x_1^{0},x_2^{2},x_3^{2},x_4^{2})-u_{3}(x_{3}^{2})-u_{4}(x_{4}^{2})\leq c(x_1^{0},x_2^{2},x_3^{1},x_4^{2})-u_{3}(x_{3}^{1})-u_{4}(x_{4}^{2}).
\end{equation}
Adding the above inequalities and eliminating similar terms we get
\begin{equation}\label{Art3:40}
c_{2}(x_{2}^{1},x_{3}^{1}) + c_{3}(x_{3}^{1},x_{4}^{1}) + c_{2}(x_{2}^{2},x_{3}^{2}) + c_{3}(x_{3}^{2},x_{4}^{2})\leq c_{2}(x_{2}^{1},x_{3}^{2}) + c_{3}(x_{3}^{2},x_{4}^{1}) + c_{2}(x_{2}^{2},x_{3}^{1}) + c_{3}(x_{3}^{1},x_{4}^{2}).
\end{equation}
By Assumption \ref{Ass:1}, the above inequality is in fact equality, which implies that we must have equality in \eqref{Art3:4} and \eqref{Art3:5}. In particular, $(x_2^{1},x_3^{2},x_4^{1}) \in  M_{x_1^{0}4}$, so
 by Lemma \ref{Art3:8} we get $$D_{x_{4}}c(x_{1}^{0},x_2^{1}, x_3^{2}, x_4^{1})= Du_{4}(x_{4}^{1})=D_{x_{4}}c(x_{1}^{0},x_2^{1}, x_3^{1}, x_4^{1}),$$ or equivalently, 
$$D_{x_4}c_{3}(x_3^{2},x_4^{1}) + D_{x_4}c_{4}(x_4^{1},x_1^{0})=Du_{4}(x_{4}^{1})=D_{x_4}c_{3}(x_3^{1},x_4^{1}) + D_{x_4}c_{4}(x_4^{1},x_1^{0}).$$
The above equalities gives $D_{x_4}c_{3}(x_3^{2},x_4^{1})=D_{x_4}c_{3}(x_3^{1},x_4^{1})$, and by Assumption \ref{Ass:2}, $x_3^{1}=x_3^{2}$. Now, note that
\begin{align*}
x_4^{1}&\in \text{Argmin}\left\lbrace x_{4}\mapsto c(x_1^{0},x_2^{1},x_3^{1},x_4)-u_{4}(x_{4})\right\rbrace\\
&=\text{Argmin}\left\lbrace x_{4}\mapsto c_3(x_3^{1},x_4) + c_4(x_4, x_1^{0}) -u_{4}(x_{4})\right\rbrace\\
&=\text{Argmin}\left\lbrace x_{4}\mapsto c_3(x_3^{2},x_4) + c_4(x_4, x_1^{0}) -u_{4}(x_{4})\right\rbrace\\
&=\text{Argmin}\left\lbrace x_{4}\mapsto c(x_1^{0},x_2^{2},x_3^{2},x_4)-u_{4}(x_{4})\right\rbrace,
\end{align*}
as $(x_2^{2}, x_3^{2}, x_4^{2})\in M_{x_1^{0}4}$. Hence, $(x_1^{0},x_2^{2},x_3^{1},x_4^{1})=(x_1^{0},x_2^{2},x_3^{2},x_4^{1})\in M_{x_1^{0}4}$, and by Lemma \ref{Art3:8} we get $D_{x_{1}}c(x_1^{0},x_2^{2},x_3^{1},x_4^{1})= Du_{1}(x_{1}^{0})=D_{x_{1}}c(x_{1}^{0},x_2^{1}, x_3^{1}, x_4^{1})$; that is, 
$$D_{x_1}c_{1}(x_1^{0},x_2^{2}) + D_{x_1}c_{4}(x_4^{1},x_1^{0})=Du_{1}(x_{1}^{0})=D_{x_1}c_{1}(x_1^{0},x_2^{1}) + D_{x_1}c_{4}(x_4^{1},x_1^{0}).$$
Thus, $D_{x_1}c_{1}(x_1^{0},x_2^{2})=D_{x_1}c_{1}(x_1^{0},x_2^{1})$ and by Assumption \ref{Ass:3}, $x_2^{1}=x_2^{2}$. Finally, we clearly have $(x_2^{2},x_3^{2},x_4^{1})=(x_2^{1},x_3^{1},x_4^{1})\in M_{x_1^{0}4}$, hence applying one more time Lemma \ref{Art3:8} we get $D_{x_{1}}c(x_1^{0},x_2^{2},x_3^{2},x_4^{1})= Du_{1}(x_{1}^{0})=D_{x_{1}}c(x_{1}^{0},x_2^{2}, x_3^{2}, x_4^{2})$. It follows that $D_{x_1}c_{4}(x_4^{1},x_1^{0})=D_{x_1}c_{4}(x_4^{2},x_1^{0})$, and by Assumption \ref{Ass:3}, $x_4^{1}=x_4^{2}$. This completes the proof of the proposition.
\end{proof}
 \end{proposition}
Note that it is not hard to find costs of the form \eqref{Aert3:38}  satisfying Assumptions \ref{Ass:2} and \ref{Ass:3}.  Assumption \ref{Ass:1}, on the other hand, is less common. We proceed now to illustrate the previous proposition with an example, which can also be seen as a slight generalization of Proposition \ref{Art3:39} when $m=4$, $I_3$ is empty, $I_1=\{1,3\}$ and $I_2=\{2,4\}$. Note that the bi-linearity assumption from Proposition \ref{Art3:39} is relaxed here.
\begin{example}
For the cost \eqref{Aert3:38}, take $c_{1}(x_{1},x_{2})=f(x_1,x_2)$,   $c_{2}(x_{2},x_{3})=f(x_3,x_2)$, $c_{3}(x_{3},x_{4})=f(x_3, x_4)$ and $c_{4}(x_{4},x_{1})=f(x_1,x_4)$, where $f: \mathbb{R}^{n}\times \mathbb{R}^{n}\mapsto \mathbb{R}$ is a map satisfying:
\begin{enumerate}[label=(\roman*)]
\item \label{Art3:36} $f$ is additive with respect to the second coordinate; that is, $f(x, y + z)=f(x,y) + f(x,z)$ for every $x$ fixed.
\item \label{Art3:42} $f$ is bi-twisted; that is, the maps $y\mapsto D_x f(x,y)$ and $x\mapsto D_y f(x,y)$ are injective.
\end{enumerate}
Substituting into \eqref{Aert3:38} and using $(i)$ we get 
\begin{align}
c(x_1,x_2,x_3,x_4)&=f(x_1,x_2) + f(x_3,x_2) + f(x_3, x_4) + f(x_1,x_4)\nonumber \\
&=f(x_1, x_2 + x_4) + f(x_3, x_2 + x_4)\nonumber 
\end{align}
Now, let $(u_{1},u_2,u_3, u_{4})$ be a 4-tuple of Borel functions satisfying inequality \eqref{Art3:6}. Fix  $x_1^{0} \in  \left\lbrace x_1\in X_1: Du_1(x_1)\; \text{exists}\right\rbrace$ and let $(x_2^{1}, x_3^{1}, x_4^{1}), (x_2^{2}, x_3^{2}, x_4^{2})\in M_{x_1^{0}4}$. From Lemma \ref{Art3:8},
\begin{flalign*}
D_{x_{1}}f(x_1^{0}, x_2^{1}+x_4^{1})& =D_{x_{1}}c(x_{1}^{0},x_2^{1}, x_3^{1}, x_4^{1})\nonumber\\
&=Du_{1}(x_{1}^{0})\qquad\qquad\qquad\qquad \nonumber\\
&=D_{x_{1}}c(x_{1}^{0},x_2^{2}, x_3^{2}, x_4^{2})\nonumber\\
  &=D_{x_{1}}f(x_1^{0}, x_2^{2}+x_4^{2}).\nonumber\\
\end{flalign*}
From Assumption \ref{Art3:42}, we deduce
\begin{equation}
x_2^{1}+x_4^{1}=x_2^{2}+x_4^{2}.
\end{equation}
It follows that
\begin{align*}
c_{2}(x_{2}^{1},x_{3}^{1}) + c_{3}(x_{3}^{1},x_{4}^{1}) + c_{2}(x_{2}^{2},x_{3}^{2}) + c_{3}(x_{3}^{2},x_{4}^{2})&=f(x_{3}^{1},x_{2}^{1}) +f(x_{3}^{1},x_{4}^{1}) + f(x_{3}^{2},x_{2}^{2}) + f(x_{3}^{2},x_{4}^{2})\nonumber\\
&=f(x_{3}^{1},x_2^{1}+x_4^{1}) + f(x_{3}^{2},x_2^{2}+x_4^{2})\nonumber\\
&=  f(x_{3}^{1},x_2^{2}+x_4^{2}) + f(x_{3}^{2},x_2^{1}+x_4^{1})\nonumber\\
&= f(x_{3}^{2},x_{2}^{1}) + f(x_{3}^{2},x_{4}^{1})+ f(x_{3}^{1},x_{2}^{2})  + f(x_{3}^{1},x_{4}^{2})\nonumber\\
&=c_{2}(x_{2}^{1},x_{3}^{2}) + c_{3}(x_{3}^{2},x_{4}^{1}) + c_{2}(x_{2}^{2},x_{3}^{1}) + c_{3}(x_{3}^{1},x_{4}^{2}).
\end{align*}
Thus, Condition \ref{Ass:1} in Proposition \ref{Art3:41}   is trivially satisfied. Since Conditions \ref{Ass:2} and \ref{Ass:3} are also satisfied (by  $(ii)$), we obtain that $c$ is twisted on $c$-splitting sets with respect to $x_1$ and $x_4$.
\end{example}
The next Proposition was obtained from some of the essential ideas of Theorem 5.1 in \cite{Kim}, which provides Monge structure and uniqueness of the optimal measures in infimal convolution type examples.
\begin{proposition}
Let $m_{0}:=1< m_1<\ldots < m_n:=m$. Set $Y_{j}:=(x_{m_{j-1}+1}, \ldots, x_{m_{j}})$ and $(x_{m_{j-1}}, Y_{j}):=(x_{m_{j-1}},x_{m_{j-1}+1}, \ldots, x_{m_{j}})$, where $j=1, \ldots, n$,  and $(x_{m_0},Y_1, Y_2, \ldots, Y_n)=(x_1, \ldots, x_m)$.\\
 Consider the cost
\begin{equation}
c(x_{1}, \ldots, x_{m})=\sum_{j=1}^{n}c_{j}(x_{m_{j-1}}, Y_{j}),
\end{equation}
 and assume
\begin{enumerate}
\item  $c_{j}$  semi-concave for each $j$.
\item $c_{j}$ is twisted on $c_{j}$-splitting sets with respect to $x_{m_{j-1}}$; that is, for each $c_{j}$-splitting set $S^{j}\subseteq X_{m_{j-1}}\times \ldots \times X_{m_{j}}$ and $x_{m_{j-1}}\in \pi_{m_{j-1}}(S^{j})$, where $\pi_{m_{j-1}}:X_{m_{j-1}}\times \ldots \times X_{m_{j}}\mapsto X_{m_{j-1}}$ is the canonical projection, the map $ Y_{j}\mapsto D_{x_{m_{j-1}}}c_{j}(x_{m_{j-1}}, Y_{j})$ is injective on the subset of $S^{j}$  where $ D_{x_{m_{j-1}}}c_{j}(x_{m_{j-1}}, Y_{j})$ exists.
\end{enumerate}
Then, the cost $c(x_{1}, \ldots, x_{m})$ is twisted on $c$-splitting sets with respect to $x_1, x_{m_{1}}, \ldots, x_{m_{n-1}}$.
\end{proposition}
\begin{proof}
Fix $j\in \{1, \ldots, n\}$. Let us first prove that for every $c$-splitting set
 $S\subseteq \prod_{i=1}^{m}X_{i}$, the set $S^{j}:=\pi_{x_{m_{j-1}}\ldots  x_{m_{j}}} (S)$ is a $c_j$-splitting set on $\prod_{i=m_{j-1}}^{m_j}X_{i}$, or equivalently, a $c_j$-cyclical monotone set on $\prod_{i=m_{j-1}}^{m_j}X_{i}$, where  $\pi_{x_{m_{j-1}}\ldots  x_{m_{j}}}:X\mapsto \prod_{i=m_{j-1}}^{m_j}X_{i}$ is the canonical projection. Indeed, fix $S$ a $c$-splitting set on $X$, and let $\left\lbrace (x_{m_{j-1}}^{k}, \ldots, x_{m_{j}}^{k}) \right\rbrace_{k=1}^{p}\subseteq S^{j}$ and  $\sigma_{m_{j-1}}, \ldots, \sigma_{m_{j}}\in S_P$, where $S_P$ denotes the set of permutations of $P= \{1, \ldots, p\}$. We want to show 
\begin{equation}\label{Art3:32}
\sum_{k=1}^{p}c_{j}(x_{m_{j-1}}^{k}, Y_{j}^{k})=\sum_{k=1}^{p}c_j(x_{m_{j-1}}^{k}, \ldots, x_{m_{j}}^{k})\leq \sum_{k=1}^{p}c_j(x_{m_{j-1}}^{\sigma_{m_{j-1}}(k)}, \ldots, x_{m_{j}}^{\sigma_{m_{j}}(k)}).
\end{equation}
Note that for each $k\in P$, there are $Y_{s}^{k}=(x_{m_{s-1}+1}^{k}, \ldots,x_{m_{s}}^{k})$, $s\neq j$, such that  $(x_{1}^{k},Y_1^{k},Y_2^{k}, \ldots, Y_n^{k}) \in S $. 
 Set 
\begin{equation}\label{Art3:33}
\sigma_{i}= \begin{cases} 
      \sigma_{m_{j-1}} & \text{if}\quad 1\leq i \leq m_{j-1}\\
       \sigma_{m_{j}} & \text{if}\quad  m_{j}\leq i\leq m. \\
        \end{cases}
\end{equation}
Since $S$ is $c$-cyclically monotone we get
\begin{align}\label{Art3:29}
&\sum_{k=1}^{p}c_{j}(x_{m_{j-1}}^{k}, Y_{j}^{k})+\sum_{s=1}^{j-1}\sum_{k=1}^{p}c_{s}(x_{m_{s-1}}^{k}, Y_{s}^{k})+\sum_{s=j+1}^{n}\sum_{k=1}^{p}c_{s}(x_{m_{s-1}}^{k}, Y_{s}^{k})\nonumber\\
&=\sum_{s=1}^{n}\sum_{k=1}^{p}c_{s}(x_{m_{s-1}}^{k}, Y_{s}^{k})\nonumber\\
&=\sum_{k=1}^{p}c(x_1^{k},\ldots, x_m^{k})\nonumber\\
&\leq  \sum_{k=1}^{p}c(x_1^{\sigma_1(k)},\ldots, x_m^{\sigma_m(k)})\nonumber\\
&=\sum_{k=1}^{p}c_{j}(x_{m_{j-1}}^{\sigma_{m_{j-1}}(k)},x_{m_{j-1}+1}^{\sigma_{m_{j-1}+1}(k)}, \ldots,x_{m_{j}}^{\sigma_{m_{j}}(k)})+\sum_{s=1}^{j-1}\sum_{k=1}^{p}c_{s}(x_{m_{s-1}}^{\sigma_{m_{s-1}}(k)},x_{m_{s-1}+1}^{\sigma_{m_{s-1}+1}(k)}, \ldots,x_{m_{s}}^{\sigma_{m_{s}}(k)})\nonumber\\
&+\sum_{s=j+1}^{n}\sum_{k=1}^{p}c_{s}(x_{m_{s-1}}^{\sigma_{m_{s-1}}(k)},x_{m_{s-1}+1}^{\sigma_{m_{s-1}+1}(k)}, \ldots,x_{m_{s}}^{\sigma_{m_{s}}(k)})
\end{align}
From \eqref{Art3:33} we have 
\begin{align}\label{Art3:30}
\sum_{s=1}^{j-1}\sum_{k=1}^{p}c_{s}(x_{m_{s-1}}^{\sigma_{m_{s-1}}(k)},x_{m_{s-1}+1}^{\sigma_{m_{s-1}+1}(k)}, \ldots,x_{m_{s}}^{\sigma_{m_{s}}(k)})&=\sum_{s=1}^{j-1}\sum_{k=1}^{p}c_{s}(x_{m_{s-1}}^{\sigma_{m_{j-1}}(k)},x_{m_{s-1}+1}^{\sigma_{m_{j-1}}(k)}, \ldots,x_{m_{s}}^{\sigma_{m_{j-1}}(k)})\nonumber\\ 
&=\sum_{s=1}^{j-1}\sum_{k=1}^{p}c_{s}(x_{m_{s-1}}^{k},x_{m_{s-1}+1}^{k}, \ldots,x_{m_{s}}^{k})\nonumber\\
&=\sum_{s=1}^{j-1}\sum_{k=1}^{p}c_{s}(x_{m_{s-1}}^{k}, Y_{s}^{k}),
\end{align}
\begin{align}\label{Art3:31}
\sum_{s=j+1}^{n}\sum_{k=1}^{p}c_{s}(x_{m_{s-1}}^{\sigma_{m_{s-1}}(k)},x_{m_{s-1}+1}^{\sigma_{m_{s-1}+1}(k)}, \ldots,x_{m_{s}}^{\sigma_{m_{s}}(k)})&=\sum_{s=j+1}^{n}\sum_{k=1}^{p}c_{s}(x_{m_{s-1}}^{\sigma_{m_{j}}(k)},x_{m_{s-1}+1}^{\sigma_{m_{j}}(k)}, \ldots,x_{m_{s}}^{\sigma_{m_{j}}(k)})\nonumber\\
&=\sum_{s=j+1}^{n}\sum_{k=1}^{p}c_{s}(x_{m_{s-1}}^{k},x_{m_{s-1}+1}^{k}, \ldots,x_{m_{s}}^{k})\nonumber\\
&=\sum_{s=j+1}^{n}\sum_{k=1}^{p}c_{s}(x_{m_{s-1}}^{k}, Y_{s}^{k}).
\end{align}
Substituting the above equalities into inequality \eqref{Art3:29} we get \eqref{Art3:32}; that is, $S^{j}$ is a $c_j$-splitting set on $\prod_{i=m_{j-1}}^{m_j}X_{i}$.\\
Now, let $(u_1,\ldots, u_m)$ be an $m$-tuple of $c$-splitting functions for $S$ and fix $x_{1}^{0}\in \pi_{1}(S)$. Assume $D_{x_{1}}c(x_{1}^{0},x_2^{1},\ldots, x_{m}^{1})$ and $ D_{x_{1}}c(x_{1}^{0},x_2^{2},\ldots, x_{m}^{2})$ exist, and
$$ D_{x_{1}}c(x_{1}^{0},x_2^{1},\ldots, x_{m}^{1})= D_{x_{1}}c(x_{1}^{0},x_2^{2},\ldots, x_{m}^{2}), $$
where $(x_{2}^{1},\ldots, x_{m}^{1}), (x_{2}^{2},\ldots, x_{m}^{2})\in W_{x_{1}^{0}m_1\ldots m_{n-1}}$. Since $c_j$ does not depend on $x_{1}$ for every $j\in \{2, \ldots, n\}$, we immediately get $$ D_{x_{1}}c_1(x_{1}^{0},x_2^{1},\ldots, x_{m_1}^{1})= D_{x_{1}}c_1(x_{1}^{0},x_2^{2},\ldots, x_{m_1}^{2}), $$
then 
\begin{equation}\label{Art3:59}
x_{j}^{1}=x_{j}^{2}\;\;\text{for every}\;\; j\in \{2,\ldots, m_1\},
\end{equation}
 as clearly $(x_{1}^{0},x_2^{1},\ldots, x_{m_1}^{1}), (x_{1}^{0},x_2^{2},\ldots, x_{m_1}^{2})\in S^{1}$ and $c_1$ is twisted on the $c_1$-splitting set $S^1$.  In particular, $x_{m_1}^{1}=x_{m_1}^{2}$ and by Lemma \ref{Art3:8},
$$ D_{x_{m_1}}c(x_{1}^{0},x_2^{1},\ldots, x_{m}^{1})=Du_{m_1}(x_{m_1}^{1})= Du_{m_1}(x_{m_1}^{2})= D_{x_{m_1}}c(x_{1}^{0},x_2^{2},\ldots, x_{m}^{2}) $$
(here the differentiability of $u_{m_1}$ at $x_{m_1}^{1}$ follows from the fact that $(x_{2}^{1},\ldots, x_{m}^{1}),(x_{2}^{2},\ldots, x_{m}^{2})\in W_{x_{1}^{0}m_1\ldots m_{n-1}}$).
Hence,
 $$ D_{x_{m_1}}c_1(x_{1}^{0},x_2^{1},\ldots, x_{m_1}^{1}) + D_{x_{m_1}}c_2(x_{m_1}^{1}, \ldots, x_{m_2}^{1})=  D_{x_{m_1}}c_1(x_{1}^{0},x_2^{2},\ldots, x_{m_1}^{2}) + D_{x_{m_1}}c_2(x_{m_1}^{2}, \ldots, x_{m_2}^{2}).$$
Combining this with \eqref{Art3:59} we get 
$$D_{x_{m_1}}c_2(x_{m_1}^{1},x_{m_1 + 1}^{1} \ldots, x_{m_2}^{1})= D_{x_{m_1}}c_2(x_{m_1}^{1},x_{m_1 + 1}^{2} \ldots, x_{m_2}^{2}).$$
Since $c_2$ is twisted on the $c_2$-splitting set $S^{2}$ and  $(x_{m_1}^{1},x_{m_1 + 1}^{1} \ldots, x_{m_2}^{1}), (x_{m_1}^{1},x_{m_1 + 1}^{2} \ldots, x_{m_2}^{2})\in S^{2}$, we deduce 
$x_{j}^{1}=x_{j}^{2}$ for every $j\in \{m_1+1,\ldots, m_2\}$. Note that this is an iterative process, so continuing with this inductive reasoning we get $x_{j}^{1}=x_{j}^{2}$ for every $j\in \{2,\ldots, m\}$. This completes the proof of the proposition.
\end{proof}

In the following proposition, for a given subset $Y:=\{x_{t_1}, \ldots,x_{t_{s}}\}\subseteq V=\{x_{1},\ldots, x_{m}\}$ with $t_1< \ldots < t_s$ and $x\in V\setminus Y$, we will write $(Y,x):=(x_{t_{1}}, \ldots,x_{t_{s}},x)$ and $(X^{k},x^{k}):=(x_{t_{1}}^{k}, \ldots,x_{t_{s}}^{k},x^{k})$, $k=1,2$. 
\begin{proposition}
Fix $s\in\{2, \ldots,m-1\}$. Consider a sequence $\{t_{\alpha}\}_{\alpha=1}^{m-(s+1)}$ and sets $Y_2, \ldots, Y_{m-s +1}$ such that $x_{t_{\alpha}}\in Y_{\alpha +1}$, $\alpha=1, \ldots, m-(s+1)$ and $Y_{j}\subseteq \{x_{2}, \ldots, x_{s+j-2}\}\setminus \{x_{t_{\alpha}}\}_{\alpha=1}^{j-2}$ for every $j=2, \ldots, m-s+1$. Consider the cost
\begin{equation}
c(x_{1}, \ldots, x_{m})= c_{1}(x_{1},\ldots,x_{s})+\sum_{j=2}^{m-s+1}c_{j}(Y_{j}, x_{s+j-1})
\end{equation}
where   $c_{j}$ is semi-concave for each $j$, and suppose 
\begin{enumerate}
\item \label{Art3:13} $c_1$ is twisted on $\pi_{1,\ldots, s}(S)$ for every $c$-splitting set $S$, where $\pi_{1,\ldots, s}:\prod_{i=1}^{m}X_i \mapsto \prod_{i=1}^{s}X_i$ is the canonical projection; that is, for every $c$-splitting set $S$ and $x_1^{0} \in \pi_{1}(S)$, the map 
$$(x_2, \ldots, x_s)\mapsto  D_{x_{1}}c_1(x_{1}^{0},x_{2}, \ldots, x_{s}) $$ is injective on $\{(x_2, \ldots, x_s): (x_1^{0}, x_2, \ldots, x_s)\in \pi_{1,\ldots, s}(S)\}$.
\item \label{Art3:11} $c_{j}$ is $(x_{t_{j-1}}, x_{s+j-1})$ twisted for all $j=2, \ldots, m-s+1$; that is, the map $ x_{s+j-1}\mapsto D_{x_{t_{j-1}}}c_{j}(Y_{j}, x_{s+j-1})$ is injective on the subset of $X_{s+j-1}$  where $D_{x_{t_{j-1}}}c_{j}(Y_{j}, x_{s+j-1})$ exists,  for every $j=2, \ldots, m-s+1$ and $Y_{j}$ fixed.
\end{enumerate}
Then, $c$ is twisted on $c$-splitting sets with respect to the variables $x_1, x_{t_1}, \ldots, x_{t_{m-s}}$.
\end{proposition}
\begin{proof}
 Let $S\subseteq X_{1}\times \ldots \times X_{m}$ be a  $c$-splitting set and $(u_{1},\ldots, u_{m})$ an $m$-tuple  of $c$-splitting functions for $S$. Fix $x_1^{0} \in \pi_{1}(S)$ and assume $ D_{x_{1}}c(x_{1}^{0},x_{2}^{1}, \ldots, x_{m}^{1})=D_{x_{1}}c(x_{1}^{0},x_{2}^{2}, \ldots, x_{m}^{2})$, where $(x_{2}^{1}, \ldots, x_{m}^{1})$, $(x_{2}^{2}, \ldots, x_{m}^{2}) \in W_{x_1^{0}, t_1, \ldots,t_{m-s} }$. We want to show that $x_j^{1}=x_j^{2}$ for every $j=2, \ldots,m$. Indeed, since the costs $c_2, \ldots, c_{ m-s+1}$ do not depend on $x_1$, we immediately get
$$ D_{x_{1}}c_1(x_{1}^{0},x_{2}^{1}, \ldots, x_{s}^{1})=D_{x_{1}}c_1(x_{1}^{0},x_{2}^{2}, \ldots, x_{s}^{2}).$$
Hence, by Assumption \ref{Art3:13} we get 
\begin{equation}\label{Art3:14}
\text{ $x_j^{1}=x_j^{2}$ for $2\leq j\leq s$.}
\end{equation}
To prove that $x_{s+j}^{1}=x_{s+j}^{2}$ for $1\leq j\leq m-s$ we use induction on $j$. For $j=1$, note that $x_{t_1}\in Y_2\subseteq \{x_2, \ldots,x_s\}$, so $x_{t_1}^{1}=x_{t_1}^{2}$ by \eqref{Art3:14}, and by Lemma \ref{Art3:8}
$$D_{x_{t_1}}c(x_{1}^{0},x_{2}^{1}, \ldots, x_{m}^{1})=Du_{t_1}(x_{t_1}^{1})=Du_{t_1}(x_{t_1}^{2})=D_{x_{t_1}}c(x_{1}^{0},x_{2}^{2}, \ldots, x_{m}^{2}).$$
Since $x_{t_1}\notin Y_j$ for $3\leq j\leq m-s+1$, we deduce 
$$D_{x_{t_1}}c_{1}(x_{1}^{0},x_{2}^{1}, \ldots, x_{s}^{1}) + D_{x_{t_1}}c_{2}(Y_2^{1},x_{s+1}^{1}) =D_{x_{t_1}}c_1(x_{1}^{0},x_{2}^{2}, \ldots, x_{s}^{2}) +  D_{x_{t_1}}c_{2}(Y_2^{2},x_{s+1}^{2}), $$
then by \eqref{Art3:14},
$$D_{x_{t_1}}c_{2}(Y_2^{1},x_{s+1}^{1}) =D_{x_{t_1}}c_{2}(Y_2^{2},x_{s+1}^{2})$$
and $Y_2^{1}=Y_2^{2}$. Consequently, we must have $x_{s+1}^{1}=x_{s+1}^{2}$, as $c_2$ is $(x_{t_1},x_{s+1})$ twisted on $c_2$-splitting sets, by Assumption \ref{Art3:11}. 
\par
Assume $x_{s+1}^{1}=x_{s+1}^{2}, \ldots, x_{s+k-1}^{1}=x_{s+k-1}^{2}$, where $1< k=j\leq m-s$. Combining this and \eqref{Art3:14} we get $x_{t_k}^{1}=x_{t_k}^{2}$, as $x_{t_k}\in Y_{k+1}\subseteq \{x_{2}, \ldots, x_{s+k-1}\}\setminus \{x_{t_1}, \ldots, x_{t_{k-1}}\}$. Then $$D_{x_{t_k}}c(x_{1}^{0},x_{2}^{1}, \ldots, x_{m}^{1})=Du_{t_k}(x_{t_k}^{1})=Du_{t_k}(x_{t_k}^{2})=D_{x_{t_k}}c(x_{1}^{0},x_{2}^{2}, \ldots, x_{m}^{2}).$$
Since $x_{t_k}\notin Y_j$ for $k+2\leq j\leq m-s+1$, we get 
\begin{equation}\label{Art3:15}
 D_{x_{t_k}}c_{1}(x_{1}^{0},x_{2}^{1}, \ldots, x_{s}^{1}) + \sum_{j=2}^{k+1} D_{x_{t_k}}c_{j}(Y_j^{1},x_{s+j-1}^{1}) =D_{x_{t_k}}c_1(x_{1}^{0},x_{2}^{2}, \ldots, x_{s}^{2}) +  \sum_{j=2}^{k+1} D_{x_{t_k}}c_{j}(Y_j^{2},x_{s+j-1}^{2}). 
\end{equation}
Now, by induction hypothesis and \eqref{Art3:14}, $D_{x_{t_k}}c_{1}(x_{1}^{0},x_{2}^{1}, \ldots, x_{s}^{1}) =D_{x_{t_k}}c_1(x_{1}^{0},x_{2}^{2}, \ldots, x_{s}^{2})$, $ D_{x_{t_k}}c_{j}(Y_j^{1},x_{s+j-1}^{1})=D_{x_{t_k}}c_{j}(Y_j^{2},x_{s+j-1}^{2})$ for every $j=2, \ldots,k$, and $Y_{k+1}^{1}=Y_{k+1}^{2}$. Hence,  \eqref{Art3:15} reduces to $$D_{x_{t_k}}c_{k+1}(Y_{k+1}^{1},x_{s+k}^{1}) =D_{x_{t_k}}c_{k+1}(Y_{k+1}
^{1},x_{s+k}^{2}).$$
We then conclude $x_{s+k}^{1}=x_{s+k}^{2}$, as $c_{k+1}$ is $(x_{t_{k}},x_{s+k})$ twisted by Assumption \ref{Art3:11}. This completes the proof of the proposition.
\end{proof}

\vspace{0.5cm}

\begin{footnotesize}
Department of Mathematical and Statistical Sciences - University of Alberta\\
Edmonton, Alberta - Canada T6G 2G1\\
\textit{Email address:} pass@ualberta.ca\\
\textit{Email address:} vargasji@ualberta.ca
\end{footnotesize}

\end{document}